\documentclass[12pt]{amsart}
\usepackage{verbatim, latexsym, amssymb, amsmath,color}
\usepackage{enumitem}
\usepackage{epsfig}

\usepackage{geometry}
\usepackage{hyperref}
\usepackage{setspace}

\newtheorem{thm}{Theorem}[section]
\newtheorem{prop}[thm]{Proposition}
\newtheorem{lem}[thm]{Lemma}
\newtheorem{cor}[thm]{Corollary}
\theoremstyle{definition}
\newtheorem{defn}{Definition}[section]
\newtheorem{exe}{Example}[section]
\newtheorem{question}{Question}[section]
\theoremstyle{remark}
\newtheorem{rem}{Remark}[section]

\begin{document}
\title[Invariant probability measures from pseudoholomorphic curves I]{Invariant probability measures from pseudoholomorphic curves I}

\author{Rohil Prasad}
\address{Department of Mathematics\\Princeton
University\\Princeton, NJ 08544}
\thanks{This material is based upon work supported by the National Science Foundation Graduate Research Fellowship Program under Grant No. DGE-1656466.}
\email{rrprasad@math.princeton.edu}

\begin{abstract}
We introduce a method for constructing invariant probability measures of a large class of non-singular volume-preserving flows on closed, oriented odd-dimensional smooth manifolds using pseudoholomorphic curve techniques from symplectic geometry. These flows include any non-singular volume preserving flow in dimension three, and autonomous Hamiltonian flows on closed, regular energy levels in symplectic manifolds of any dimension. As an application, we use our method to prove the existence of obstructions to unique ergodicity for this class of flows, generalizing results of Taubes and Ginzburg-Niche.
\end{abstract}

\maketitle

\tableofcontents

\section{Introduction}

This paper is concerned with the construction of invariant probability measures for the autonomous flow associated to a \emph{framed Hamiltonian structure} $\eta = (\lambda, \omega)$ on a closed, oriented smooth manifold $M$ of dimension $2n + 1$ for $n \geq 1$. 

\begin{defn} \label{defn:framedHamiltonianStructures}
A \textbf{framed Hamiltonian structure} on $M$ is the datum of a $1$-form $\lambda$ and a two-form $\omega$ such that:
\begin{enumerate}
    \item $\lambda \wedge \omega^n > 0$.
    \item $d\omega = 0$.
\end{enumerate}

We call a framed Hamiltonian structure \emph{exact} if $\omega$ is exact. The datum of $M$ along with a choice of framed Hamiltonian structure $\eta = (\lambda, \omega)$ is called a \textbf{framed Hamiltonian manifold}.
\end{defn}

A framed Hamiltonian manifold $M$ with framed Hamiltonian structure $\eta = (\lambda, \omega)$ is denoted by the tuple $(M, \eta)$. The autonomous flow in question is the flow of the \emph{Hamiltonian vector field} $X$ associated to $\eta$. 

\begin{defn} \label{defn:hamiltonianVectorField}
The \textbf{Hamiltonian vector field} associated to a framed Hamiltonian structure $(\lambda, \omega)$ is the unique vector field $X$ satisfying
$$\lambda(X) \equiv 1$$
and
$$\omega(X, -) \equiv 0.$$
\end{defn}

We present three broad examples of natural dynamical situations which can be modeled by a framed Hamiltonian structure and its associated flow. 

\begin{exe} \label{exe:divergenceFreeVectorFields}
Let $(M^3, g)$ be a closed, Riemannian three-manifold with volume form $\text{dvol}_g$. Let $X$ be any non-singular, volume-preserving vector field. Then we may define a framed Hamiltonian structure by taking $\lambda$ to be the dual of $X$ with respect to the Riemannian metric $g$, and $\omega$ to be the contraction $\text{dvol}_g(X, -)$. 

Note that $\lambda \wedge \omega = \text{dvol}_g$, and $\omega$ is closed by virtue of Cartan's formula and the fact that $X$ is volume-preserving. Furthermore, $X$ is the Hamiltonian vector field associated to $\eta = (\lambda, \omega)$. 
\end{exe}

\begin{exe} \label{exe:hypersurfaces}
Let $(W,\Theta)$ be a symplectic manifold of dimension $2n$. Let $M$ be a closed hypersurface in $W$, equal to $H^{-1}(0)$ for some smooth Hamiltonian function
$$H: W \to \mathbb{R}.$$

We assume that $0$ is a regular value of $H$, so $M$ is a smooth manifold. The Hamiltonian vector field $X_H$ associated to $H$ is the unique vector field on $W$ such that
$$\Theta(X_H, -) = -dH(-).$$

It is a quick exercise to show, using the Cartan formula for the Lie derivative, that the function $H$ is invariant under the flow of $X_H$, and that $X_H$ is tangent to $M$ at every point of $M$. 

Pick an almost-complex structure $J$ on $W$ \emph{compatible} with the symplectic form $\Theta$, meaning that 
$$\Theta(V, JV) > 0$$
for any nonzero tangent vector $V$ and
$$\Theta(J-, J-) = \Theta(-, -).$$

Let $\lambda$ be the one-form on $M$ with kernel $\xi = TM \cap J(TM)$. Set $\omega = \Theta|_M$. Then $\eta = (\lambda, \omega)$ is a framed Hamiltonian structure on $M$. The associated Hamiltonian vector field $X$ can be shown to coincide with the restriction of $X_H$ to $M$. 
\end{exe}

\begin{exe} \label{exe:contactManifolds}
Any contact manifold $M$ equipped with a choice of contact form $\lambda$ has a natural framed Hamiltonian structure, defined by $\eta = (\lambda, d\lambda)$. Furthermore, the Reeb vector field $R$ associated to $\lambda$ is equal to the framed Hamiltonian vector field $X$ associated to $\eta$. 
\end{exe}

\begin{rem}
A closed two-form $\omega$ such that there \emph{exists} a $\lambda$ as in Definition \ref{defn:framedHamiltonianStructures} is commonly known as an ``odd-symplectic structure''.

A framed Hamiltonian structure could also be seen as a more general version of a ``stable Hamiltonian structure'' as in \cite{BEHWZ03}. A stable Hamiltonian structure is the datum of a framed Hamiltonian structure, but with the additional assumption that
$$\text{ker}(\omega) \subseteq \text{ker}(d\lambda).$$
\end{rem}

Fix a framed Hamiltonian manifold $(M, \eta = (\lambda, \omega))$ of dimension $2n+1$ and let $X$ denote its associated Hamiltonian vector field. We make two observations regarding the problem of constructing invariant probability measures. 

First, it is a consequence of Cartan's formula that flow of $X$ preserves the natural volume form 
$$\text{dvol}_\eta = \lambda \wedge \omega^n.$$

Second, it is not difficult to find a probability measure invariant under the flow of $X$. Write $\phi_X^T$ for the time $T$ flow of $X$. Pick any point $p \in M$. Then we can pick a sequence $T_k \to \infty$ and define a sequence $\mu_k$ of probability measures by averaging over the time $T_k$ flow-lines starting at $p$, namely 
$$\int f d\mu_k = \frac{1}{T_k} \int_0^{T_k} f(\phi_X^t(p)) dt .$$

After passing to a subsequence, the measures $\mu_k$ converge weakly to an invariant probability measure $\mu$. 

However, it is difficult to tell without additional assumptions on the flow that this construction yields any ``interesting'' invariant measures. No matter the choices made in the above construction, it could be the case that $\mu$ is merely the normalized volume density 
$$(\int_M \lambda \wedge \omega^n)^{-1} \text{dvol}_\eta.$$

Following this line of thought to the most extreme case, if the \emph{only} invariant probability measure is the normalized volume density above then the flow is called ``uniquely ergodic''. In this context, it is useful to keep in mind the following theorem, proved by Ginzburg-Niche in \cite{GinzburgNiche2015} using McDuff's contact type criterion \cite{McDuff87}. 

\begin{thm} \cite[Theorem $1.1$]{GinzburgNiche2015}
Let $(M, \eta = (\lambda, \omega))$ be a closed, oriented framed Hamiltonian manifold of dimension $2n+1$ for $n \geq 1$. If $\omega$ is exact, the ``self-linking number'', defined as 
$$\int_M \nu \wedge \omega^n$$
for any choice of primitive $\nu$ of $\omega$ is nonzero, and the Hamiltonian flow is uniquely ergodic, then the framed Hamiltonian structure is ``contact type''. That is, there is a framed Hamiltonian structure $\eta' = (\lambda', \omega)$ such that $\lambda'$ is a contact form and $\omega = d\lambda'$. 
\end{thm}

Observe that replacing $\lambda$ in a framed Hamiltonian structure with another one-form $\lambda'$ merely multiplies the Hamiltonian vector field by a positive smooth function, so the flows coincide up to a time change. 

The above theorem, given a mild topological assumption, reduces the study of uniquely ergodic framed Hamiltonian flows to Reeb flows (see Example \ref{exe:contactManifolds}). Observe that in the case where $M$ is a closed regular energy level in $\mathbb{R}^{2n}$ equipped with the standard symplectic form (see Example \ref{exe:hypersurfaces}), the self-linking number is equal to the symplectic volume of the compact region bounded by $M$, so it is always nonzero. 

Furthermore, in dimension three, Taubes' proof of the Weinstein conjecture \cite{TaubesWeinstein} shows that a Reeb flow always has a periodic orbit. This implies that Hamiltonian flows in dimension three with $\omega$ exact and nonzero self-linking number cannot be uniquely ergodic. If the Weinstein conjecture were true for all higher-dimensional contact manifolds, then the same argument would show that any framed Hamiltonian flow in any dimension with $\omega$ exact and nonzero self-linking number cannot be uniquely ergodic.

It was also shown by Taubes in \cite{TaubesVectorFields} using Seiberg-Witten theory, in a manner similar to his proof of the Weinstein conjecture in dimension $3$, that any framed Hamiltonian flow on a closed $3$-manifold with $\omega$ exact and nonzero self-linking number cannot be uniquely ergodic. 

The construction of interesting invariant probability measures also has potential applications beyond the question of unique ergodicity. One major application could be answering (in the negative) a question of Herman from the $1998$ International Congress of Mathematicians \cite{Herman1998}:

\begin{question}
\label{question:herman} When $n \geq 2$, can one find a $C^\infty$ compact, connected regular hypersurface in $\mathbb{R}^{2n}$ on which the characteristic flow is minimal?
\end{question}

Recall that a vector field on a manifold is ``minimal'' if and only if the only non-empty closed subset invariant under its flow is the entire manifold. Using the construction in Example \ref{exe:hypersurfaces}, one finds that Question \ref{question:herman} is equivalent to asking if the Hamiltonian vector field $X$ induced on any closed regular hypersurface $M \subset \mathbb{R}^{2n}$ by this construction has no non-empty closed invariant subsets other than the entire manifold $M$. 

Now recall that the support of an invariant probability measure is a non-empty closed invariant set. Therefore, if one can find, for any closed regular hypersurface $M \subset \mathbb{R}^{2n}$, an $X$-invariant probability measure that does not have support equal to the entire manifold $M$, then they have answered Herman's question in the negative. 

On the other hand, by combining our Theorem \ref{thm:invariantMeasures} and the existence theorem for pseudoholomorphic curves proved in the sequel \cite{sequel}, it is known that all such hypersurfaces admit explicit $X$-invariant probability measures that are not equal to the normalized volume measure. Any invariant probability measure has support equal to a non-empty, closed invariant subset. Therefore, if we can show that the invariant probability measures given by Theorem \ref{thm:invariantMeasures} do not have full support, we have answered Question \ref{question:herman} in the negative. 

Question \ref{question:herman} has recently been answered in the negative for closed regular hypersurfaces in $\mathbb{R}^4$ (the case where $n = 2$) in a groundbreaking new paper by Fish-Hofer \cite{FishHoferFeral}. Their proof uses a new class of pseudoholomorphic curves called ``feral'' pseudoholomorphic curves, which are similar but not identical to the pseudoholomorphic curves used in this work. 

We present in our main theorem, Theorem \ref{thm:invariantMeasures}, a construction of invariant probability measures of the Hamiltonian vector field $X$ using the theory of pseudoholomorphic curves from symplectic geometry. 

The utility of this construction is three-fold. First, we can leverage the geometric and topological properties of pseudoholomorphic curves to conclude, under mild assumptions that are topological in nature, that our measures are ``interesting''. That is, they are not equal to the normalized volume density, and depending on the situation they satisfy additional properties (see Theorems \ref{thm:mainNonExact}, \ref{thm:mainExact}, and \ref{thm:confoliations}). 

Second, our method is flexible enough to apply, in contrast to the theorems stated in \cite{GinzburgNiche2015} and \cite{TaubesVectorFields}, to the case where $\omega$ is \emph{not} exact. The Seiberg-Witten theory approach of \cite{TaubesVectorFields} could be attempted in the non-exact case, but even then it cannot be used in dimensions higher than three. 

Third, the usage of pseudoholomorphic curves connects the question of finding interesting invariant probability measures to the question of finding interesting moduli spaces of pseudoholomorphic curves, which can be investigated by appealing to rich symplectic invariants such as Gromov-Witten theory and symplectic field theory. See \cite{McduffSalamon12} and \cite{EGH00} respectively for overviews of these invariants. We make use of Gromov-Witten theory and Seiberg-Witten theory in the sequel \cite{sequel} to construct examples of pseudoholomorphic curves to which the results of this paper can be applied.

We also elaborate more on the potential application of this result to Herman's Question \ref{question:herman}. By combining our Theorem \ref{thm:invariantMeasures} and the existence theorem for pseudoholomorphic curves provided in \cite{sequel}, we can explicitly construct on any closed, regular hypersurface $M \subset \mathbb{R}^{2n}$ an $X$-invariant probability measure that is not equal to the normalized volume density. As mentioned in our prior discussion regarding Question \ref{question:herman}, if we can show that the invariant probability measures given by Theorem \ref{thm:invariantMeasures} do not have full support, we have answered Question \ref{question:herman} in the negative. 

\begin{rem}
We also take the opportunity to make an interesting technical remark regarding the types of pseudoholomorphic curves that we use in our construction of invariant measures. The pseudoholomorphic curves that we consider satisfy much less stringent topological and geometric assumptions than the types considered in the symplectic invariants referred to above. 

For example, it is possible that the images of the curves in the target do not have a uniform bound (in $r$) on their area over all balls of radius $r$ in the target. It is also possible that the domains of the curves have infinite genus, infinitely many non-compact ends, and infinitely many marked and nodal points. These latter three properties make our assumptions even weaker than those made of the ``feral'' pseudoholomorphic curves introduced in \cite{FishHoferFeral}. We have yet to encounter, however, an example where the domains of the curves do in fact have more complicated topology than feral pseudoholomorphic curves. 
\end{rem}

Before we state our main theorems in Section \ref{subsec:thmStatements}, we must first write down several definitions and notations in Sections \ref{subsec:jCurves} and \ref{subsec:adaptedCylinders}. We introduce in these sections the notion of a ``$\omega$-finite'' pseudoholomorphic curve in an adapted cylinder over a framed Hamiltonian manifold $M$, which is the main object of study in this work.

Before doing so, we present a schematic overview of the main results. Theorem \ref{thm:invariantMeasures} presents our construction of invariant measures from pseudoholomorphic curves. The statement is technical and requires everything from Sections \ref{subsec:jCurves} and \ref{subsec:adaptedCylinders} to understand fully. 

Theorems \ref{thm:mainNonExact} and \ref{thm:mainExact} present, under mild topological assumptions, some structural properties of the invariant measures provided by Theorem \ref{thm:invariantMeasures}. Both of these theorems imply that, under the aforementioned assumptions, the flow of the Hamiltonian vector field is not uniquely ergodic. 

Theorem \ref{thm:confoliations} gives finer information about the support of the invariant measures constructed in Theorem \ref{thm:invariantMeasures} if the flow is ``confoliation-type''. It shows that the support of these measures lie inside the zero set of a smooth, non-negative function. 

If the reader is willing to accept Theorem \ref{thm:invariantMeasures} as a black box, then they do not need to read Sections \ref{subsec:jCurves} and \ref{subsec:adaptedCylinders} to understand the statements of Theorems \ref{thm:mainNonExact}, \ref{thm:mainExact}, and \ref{thm:confoliations}. The \emph{proofs} of any of the results listed above, however, use definitions and notations from Sections \ref{subsec:jCurves} and \ref{subsec:adaptedCylinders} constantly. 

For the remainder of this paper, we assume that all manifolds defined are oriented.

\subsection{Pseudoholomorphic curves in almost-complex manifolds} \label{subsec:jCurves}

\begin{defn}\label{defn:markedNodalSurface}
A \textbf{nodal} Riemann surface is the datum 
$$(C, j, D)$$
which consists of the following data:
\begin{itemize}
    \item A (possibly disconnected) smooth surface $C$.
    \item An almost-complex structure $j$ on $C$.
    \item A set of pairs of \emph{nodal points} $D = \{(\overline{z}_1, \underline{z}_1), (\overline{z}_2, \underline{z}_2), \ldots \}$ such that the full collection of points in $D$ is discrete, closed, and no point of $D$ lies in the boundary $\partial C$.
\end{itemize}

A \textbf{marked nodal} Riemann surface is the datum 
$$(C, j, D, \mu)$$
where $(C, j, D)$ is a nodal Riemann surface and $\mu$ is a closed, discrete subset of $C \setminus (D \cup \partial C)$, called the set of \emph{marked points}. 
\end{defn}

\begin{rem}
We will often abuse notation and refer to the set $D$ of \emph{pairs} of nodal points as merely a \emph{set} of nodal points.  
\end{rem}

Recall for any connected \emph{compact} surface $C$, the genus of $C$, denoted by $\text{Genus}(C)$, is the genus of the closed surface obtained from $C$ by capping off all of the boundary components of $C$ with disks. For any disconnected surface $C$ with compact connected components, $\text{Genus}(C)$ takes values in the extended positive integers $\mathbb{Z}_{\geq 0} \cup \{\infty\}$. We do not restrict $C$ to have finitely many connected components. If infinitely many connected components of $C$ have positive genus, we set $\text{Genus}(C) = \infty$. If not, $\text{Genus}(C)$ is equal to the sum of the genera of its connected components. 

In the case where $C$ is not compact, the genus is defined using a compact exhaustion. 

\begin{defn} \label{defn:curveGenus}
The \textbf{genus} $\text{Genus}(C)$ of a connected, non-compact Riemann surface $C$ is defined as follows. Select a sequence of compact sub-surfaces $\{C_k\}_{k \geq 0}$ such that
$$C_k \subset \text{int}(C_{k+1})$$
for every $k$ and 
$$C = \cup_{k \geq 0} C_k.$$

Then set
$$\text{Genus}(C) = \lim_k \text{Genus}(C_k) \in \mathbb{Z}_{\geq 0} \cup \{\infty\}.$$

If $C$ is disconnected and non-compact, and any of its connected components $C'$ have $\text{Genus}(C') = \infty$, we set $\text{Genus}(C) = \infty$. Otherwise, we define $\text{Genus}(C)$ to be equal to $\infty$ if infinitely many connected components have positive genus, and equal to the sum of the genera of its connected components if only finitely many connected components have positive genus. 
\end{defn}

It is a straightforward exercise to show that the definition of $\text{Genus}(C)$ is independent of the choice of compact exhaustion $\{C_k\}$. 

Now we can define pseudoholomorphic curves.

\begin{defn} \label{defn:nodalCurve}
A \textbf{(marked nodal) pseudoholomorphic curve} or \textbf{(marked nodal) $J$-holomorphic curve} is the datum
$$\mathbf{u} = (u, C, j, W, J, D, \mu)$$
which consists of the following data:
\begin{itemize}
    \item A marked nodal Riemann surface $(C, j, D, \mu)$.
    \item An almost-complex manifold $(W, J)$.
    \item A smooth map
    $$u: C \to W$$
    satisfying the non-linear Cauchy-Riemann equation
    $$du + J \circ du \circ j = 0$$
    and the property that $u(\overline{z}) = u(\underline{z})$ for every pair $(\overline{z}, \underline{z}) \in D$.
\end{itemize}
\end{defn}

In symplectic geometry, one usually studies pseudoholomorphic curves
$$\mathbf{u} = (u, C, j, W, J, D, \mu)$$
in symplectic manifolds $(W, \Omega)$, where $J$ is either \emph{tame} or \emph{compatible} with respect to the symplectic form $\Omega$.

\begin{defn}
\label{defn:tamedAndCompatible} An almost-complex structure $J$ on a symplectic manifold $(W, \Omega)$ is \textbf{tame} if 
$$\Omega(V, JV) > 0$$
for any tangent vector $V$. It is \textbf{compatible} if it is tame and additionally
$$\Omega(JV_1, JV_2) = \Omega(V_1, V_2)$$
for any pair of tangent vectors $V_1$ and $V_2$. 
\end{defn}

It is also useful to write down the following definition. 

\begin{defn} \label{defn:connectedCurve}
A marked nodal pseudoholomorphic curve 
$$\mathbf{u} = (u, C, j, W, J, D, \mu)$$
is \textbf{connected} if the surface constructed from $C$ by performing connect sums at each pair of nodal points is connected.
\end{defn}

\subsection{Pseudoholomorphic curves in adapted cylinders} \label{subsec:adaptedCylinders}

Fix $M$ to be a closed, smooth manifold of dimension $2n+1$. Equip $M$ with a framed Hamiltonian structure $\eta = (\lambda, \omega)$ and let $X$ be the associated Hamiltonian vector field. 

Before proceeding, we write down a couple of basic objects associated to $(M, \eta)$. 

There is a natural codimension-one tangent distribution on $M$ given by
$$\xi = \text{ker}(\lambda)$$
and the pair $(\xi, \omega)$ forms the datum of a symplectic bundle. 

Write $\langle X \rangle$ for the one-dimensional subbundle of $TM$ defined by the span of the Hamiltonian vector field $X$. Then there is a splitting
$$TM = \langle X \rangle \oplus \xi.$$

We will define the projection map
$$\pi_\xi: TM \to \xi$$
to be the projection with kernel $\langle X \rangle$.

Now recall the definition of an almost-Hermitian manifold.

\begin{defn} \label{defn:almostHermitian}
An \textbf{almost-Hermitian structure} on an even-dimensional manifold $W$ is the datum of a Riemannian metric $g$ and an almost-complex structure $J$ such that 
$$g(J-, J-) = g(-, -).$$

A \textbf{almost-Hermitian} manifold is an even-dimensional manifold $W$ equipped with an almost-Hermitian structure $(J,g)$.
\end{defn}

There is a distinguished class of translation-invariant almost-complex structures $J$ on the cylinder $\mathbb{R} \times M$ called \emph{$\eta$-adapted} almost-complex structures which incorporate the Hamiltonian vector field $X$. 

Such an almost-complex structure (denoted for the moment by $J$) has a natural associated Riemannian metric $g$, and it is easy to show that the pair $(J,g)$ forms an almost-Hermitian structure on $\mathbb{R} \times M$. 

We now proceed with all of the relevant definitions. 

\begin{defn} \label{defn:adaptedJ}
An almost-complex structure on $\mathbb{R} \times M$ is \textbf{$\eta$-adapted} (or just \textbf{adapted} when the context is clear) if it satisfies the following properties:
\begin{enumerate}
    \item $J$ is invariant with respect to translation in the $\mathbb{R}$ factor. 
    \item Let $a$ denote the $\mathbb{R}$ coordinate. Then
    $$J(\partial_a) = X$$
    and
    $$J(X) = -\partial_a$$
    \item $J$ preserves $\xi = \text{ker}(\lambda)$ and the restriction $J|_\xi$ is compatible with $\omega$, i.e.
    $$\omega(JV_1, JV_2) = \omega(V_1,V_2)$$
    for any pair of tangent vectors $V_1$ and $V_2$ in $\xi$ and
    $$\omega(V, JV) > 0$$
    for any tangent vector $V$ in $\xi$. 
\end{enumerate}
\end{defn}

Now fix an adapted almost-complex structure $J$ on $\mathbb{R} \times M$. We will always use $a$ to denote the $\mathbb{R}$-coordinate on $\mathbb{R} \times M$. Then we may define a Riemannian metric $g$ by 
$$g(-, -) = (da \wedge \lambda + \omega)(-, J-).$$

Here $\lambda$ and $\omega$ denote the translation-invariant pullbacks to $\mathbb{R} \times M$ of the corresponding differential forms on $M$. 

\begin{rem} \label{rem:metric}
Note by definition that this metric is cylindrical, so it induces a natural metric $g$ on $M$ given explicitly by
$$g(-, -) = (\lambda \otimes \lambda)(-, -) + \omega(-, J-).$$

When we think of $g$ as a metric on $M$ in what follows, it will be this metric. 

It follows from this computation and the previous definitions that
$$\text{dvol}_g = \frac{1}{n!}\lambda \wedge \omega^n$$
and $X$ preserves this volume form. 
\end{rem}

The proof of the following lemma amounts to a quick algebraic computation using the above definitions.

\begin{lem}
The datum $(J,g)$ defines an almost-Hermitian structure on $\mathbb{R} \times M$. 
\end{lem}

We refer to almost-Hermitian manifolds of the form $(\mathbb{R} \times M, J, g)$ as \emph{adapted cylinders} over $M$. 

\begin{defn}
\label{defn:adaptedCylinder} An \textbf{$\eta$-adapted cylinder} (or \textbf{adapted cylinder} when the context is clear) over a framed Hamiltonian manifold $(M, \eta = (\lambda, \omega))$ is an almost-Hermitian manifold of the form $(\mathbb{R} \times M, J, g)$ where $(J,g)$ is the almost-Hermitian structure defined by a choice of an $\eta$-adapted almost-complex structure $J$ and 
$$g(-, -) = (da \wedge \lambda + \omega)(-, J-).$$
\end{defn}

\begin{rem}
The definition of an $\eta$-adapted almost-complex structure $J$ on $\mathbb{R} \times M$ was introduced in \cite{FishHoferFeral}, and mirrors the definition of an admissible almost-complex structure $J$ on the symplectization $\mathbb{R} \times M$ over a contact manifold introduced in \cite{HoferWeinstein}. Accordingly, \cite{FishHoferFeral} refers to $\mathbb{R} \times M$ as the ``symplectization'' due to this correspondence with the contact case. 

We prefer to use the terminology ``adapted cylinder'' because $\mathbb{R} \times M$ may not be symplectic if $M$ is not contact. More precisely, when $M$ is a framed Hamiltonian manifold that is not a contact manifold, $\mathbb{R} \times M$ to our knowledge does not necessarily admit a natural symplectic form that is compatible with an $\eta$-adapted almost-complex structure $J$. 
\end{rem}

Fix a choice of $\eta$-adapted cylinder $(\mathbb{R} \times M, J, g)$ over $M$. Now we will define ``$\omega$-finite'' pseudoholomorphic curves, which are $J$-holomorphic curves of the form
$$\mathbf{u} = (u, C, j, \mathbb{R} \times M, J, D, \mu)$$
satisfying some additional restrictions. 

We begin with defining a couple of useful notions of energy. Let $a$ denote the $\mathbb{R}$-coordinate projection function on $\mathbb{R} \times M$. Write $\mathcal{R}$ for the set of regular values of the real-valued map $a \circ u$. By Sard's theorem it follows that $\mathcal{R}$ has full Lebesgue measure in $\mathbb{R}$. 

\begin{defn} \label{defn:energy}
Let 
$$\mathbf{u} = (u, C, j, \mathbb{R} \times M, J, D, \mu)$$
be a $J$-holomorphic curve. The \textbf{$\lambda$-energy} of $\mathbf{u}$ is the function
$$E_\lambda(u, -): \mathcal{R} \to \mathbb{R}$$
defined by
$$E_\lambda(u, t) = \int_{(a \circ u)^{-1}(t)} u^*\lambda.$$

The \textbf{$\omega$-energy} of $\mathbf{u}$ is
$$E_\omega(u) = \int_C u^*\omega.$$
\end{defn}

The following lemma is a consequence of the adaptedness of $J$.

\begin{lem} \label{lem:omegaNonnegative}
Let 
$$\mathbf{u} = (u, C, j, \mathbb{R} \times M, J, D, \mu)$$
be a connected $J$-holomorphic curve.

The $\omega$-energy $E_\omega(u)$ is non-negative, and it is equal to zero if and only if either the map $u$ is constant or has image contained in $\mathbb{R} \times \gamma \subset \mathbb{R} \times M$, where $\gamma$ is a trajectory of the Hamiltonian vector field.  

Moreover, $u^*\omega$ is non-negative pointwise on the tangent planes of $C$. 
\end{lem}

Now we can define $\omega$-finite pseudoholomorphic curves.

\begin{defn}\label{defn:omegaFiniteCurve}
A $J$-holomorphic curve 
$$\mathbf{u} = (u, C, j, \mathbb{R} \times M, J, D, \mu)$$ is \textbf{$\omega$-finite} if
\begin{itemize}
\item The map $u$ is proper.
\item The image $u(C)$ has finitely many connected components.
\item $E_\omega(u) < \infty$.
\item The domain $C$ has at least one non-compact connected component.
\item There is a compact set $K \subseteq \mathbb{R} \times M$ such that the image $u(\partial C)$ of the boundary of $C$ is contained in $K$.
\end{itemize}
\end{defn}

The second condition in Definition \ref{defn:omegaFiniteCurve} regarding the connectedness of $u(C)$ is unnatural, but required to rule out some pathological cases in the proof of Proposition \ref{prop:periodicOrbits}. However, it is extremely weak and quite easy to verify in practice. We also note that Definition \ref{defn:omegaFiniteCurve} has some similarities to the definition of a ``feral curve'' given in \cite{FishHoferFeral}.

However, the domain of a feral pseudoholomorphic curve satisfies additional ``finite topology'' assumptions that the domain of an $\omega$-finite curve is not required to satisfy. 

First, the genus of the domain $C$ of a feral pseudoholomorphic curve is finite. See Definition \ref{defn:curveGenus} and the discussion above it for our definition of the genus of $C$. 

Second, a feral pseudoholomorphic curve has finitely many marked and nodal points. 

Third, the domain $C$ of a feral pseudoholomorphic curve is required to have finitely many non-compact ends, referred to as ``generalized punctures''.

We define this below, but again one should either check that this is well-defined or refer to \cite{FishHoferFeral}.

\begin{defn}
For any compact sub-surface $C' \subseteq C$, write $\mathcal{C}_{\text{non-compact}}(C\setminus C')$ for the number of non-compact connected components of $C \setminus C'$. Then for any exhaustion of $C$ by compact sub-surfaces $\{C_k\}_{k \geq 0}$ as in Definition \ref{defn:curveGenus}, define
$$\text{Punct}(C) = \lim_{k \to \infty} \#\mathcal{C}_{\text{non-compact}}(C \setminus (C_k \setminus \partial C_k)) \in \mathbb{Z}_{\geq 0} \cup \{\infty\}.$$
\end{defn}

We collect this discussion in the following definition.

\begin{defn} \label{defn:feralCurve}
A $J$-holomorphic curve
$$\mathbf{u} = (u, C, j, \mathbb{R} \times M, J, D, \mu)$$
is \textbf{feral} if it is $\omega$-finite, and moreover satisfies the following conditions:
\begin{itemize}
\item $\text{Genus}(C) < \infty$.
\item $\#D < \infty$.
\item $\#\mu < \infty$.
\item $\text{Punct}(C) < \infty.$
\end{itemize}
\end{defn}

When we construct examples in the sequel \cite{sequel}, we will verify that the pseudoholomorphic curves in question are not only $\omega$-finite, but also feral. 

\subsection{Statement of main theorems} \label{subsec:thmStatements}

Now we have sufficient background to give statements of our main theorems. The statement of Theorem \ref{thm:invariantMeasures}, our main theorem about existence of invariant probability measures is fairly technical. We will give the statement below as well as various applications, and then provide an explanation of the main ideas behind Theorem \ref{thm:invariantMeasures} in Section \ref{subsec:history}. 

Basic terminology from measure theory is also used throughout the statements of the main theorems and their proofs. The reader may use Appendix \ref{sec:measureTheory} or \cite{Simon14} as a reference.

For the remainder of this subsection, fix a closed framed Hamiltonian manifold $(M^{2n+1}, \eta = (\lambda, \omega))$. 

\begin{thm}
\label{thm:invariantMeasures} Suppose that there exists an $\eta$-adapted almost-complex structure $J$ on $\mathbb{R} \times M$, a Riemann surface $(C, j)$ and an $\omega$-finite $J$-holomorphic curve 
$$\mathbf{u} = (u, C, j, \mathbb{R} \times M, J, D, \mu).$$

Then, either $X$ has a periodic orbit or the following statement holds. For any unbounded sequence $\{t_k\}$ of real numbers for which $u(C)$ has nontrivial intersection with the level sets $\{t_k\} \times M$, the following holds after possibly passing to a subsequence. 

First, there exists a sequence of real numbers $\{\delta_k\}$ converging to zero such that $u(C)$ intersects $\{t_k + \delta_k\} \times M$ for every $k$. 

Second, the Radon probability measures $\sigma_k$ on $M$ defined by
$$\int_M f d\sigma_k = (\int_{u(C) \cap \{t_k + \delta_k\} \times M} f\lambda) / (\int_{u(C) \cap \{t_k + \delta_k\} \times M} \lambda)$$
converge weakly to a non-zero probability measure $\sigma$, invariant with respect to the Hamiltonian vector field $X$. 
\end{thm}

\begin{rem}
The above theorem requires the existence of an unbounded sequence $\{t_k\}$ of real numbers for which $u(C)$ has nontrivial intersection with the level sets $\{t_k\} \times M$. 
The existence of such a sequence follows from the definition of a $\omega$-finite pseudoholomorphic curve, since $u$ is proper and $C$ has at least one non-compact connected component. 
\end{rem}

Theorem \ref{thm:invariantMeasures} explicitly states how one may associate an invariant measure $\sigma$ to a pseudoholomorphic curve in $\mathbb{R} \times M$. 

Recall that the framed Hamiltonian manifold $M$ admits a natural $X$-invariant probability measure, given by the normalized density
$$\text{dvol}_\eta = (\int_M \lambda \wedge \omega^n)^{-1}(\lambda \wedge \omega^n).$$

A proof of this fact is sketched below the statement of Lemma \ref{lem:volumePreserving}. The following two results show that, given the existence of certain simple topological obstructions, the invariant measure $\sigma$ produced by Theorem \ref{thm:invariantMeasures} is not equal to $\text{dvol}_\eta$.

\begin{thm} \label{thm:mainNonExact}
Suppose $\omega^n$ is not exact. Then any invariant measure $\sigma$ produced by Theorem \ref{thm:invariantMeasures} is not the normalized volume density $\text{dvol}_\eta$. Moreover, the one-dimensional closed, $X$-invariant current $T$ defined by 
$$T(\alpha) = \sigma(\alpha(X))$$
is null-homologous. 
\end{thm}

The second result assumes $\omega$ is exact. To state this result, recall that the \emph{self-linking number} $Lk(\omega)$ of $\omega$ is defined by picking any primitive $\nu$ of $\omega$ and setting
$$Lk(\omega) = \int_M \nu \wedge \omega^n.$$

Note that $Lk(\omega)$ does not depend on the choice of primitive. Any two such primitives differ by a closed $1$-form. Using the fact that $\omega$ is exact and performing an integration by parts indicates that the difference of the two integrals associated to these two primitives is zero. In dimension three, $Lk(\omega)$ can be interpreted as a measure of the ``asymptotic self-linking'' of the vector field $X$ (see \cite{Arnold74, ArnoldKhesin98, Vogel13}). 

\begin{thm} \label{thm:mainExact}
Suppose $\omega$ is exact and $Lk(\omega) \neq 0$. Then any invariant measure $\sigma$ produced by Theorem \ref{thm:invariantMeasures} is not the normalized volume density. 
\end{thm}

Suppose $M$ satisfies the conditions of Theorem \ref{thm:invariantMeasures}, and the conditions of either Theorem \ref{thm:mainNonExact} or Theorem \ref{thm:mainExact}. This implies that the Hamiltonian vector field $X$ has at least two non-trivial invariant probability measures: the measure $\sigma$ from Theorem \ref{thm:invariantMeasures} and the normalized volume density $\text{dvol}_\eta$. 

Recall that a vector field is \emph{uniquely ergodic} if it has exactly one invariant probability measure. It follows from the above discussion that the Hamiltonian vector field $X$ is not uniquely ergodic in many cases.

\begin{cor} \label{cor:uniqueErgodicity}
Suppose $M$ satisfies the conditions of Theorem \ref{thm:invariantMeasures}, and the conditions of either Theorem \ref{thm:mainNonExact} or Theorem \ref{thm:mainExact}. Then the Hamiltonian vector field $X$ is not uniquely ergodic.
\end{cor}

\begin{rem} \label{rem:examples}
Our pseudoholomorphic curve technique can be applied when the framed Hamiltonian manifold $M$ satisfies the assumptions of Theorem \ref{thm:invariantMeasures}. In other words, we require the existence of a general type of pseudoholomorphic curve mapping into $\mathbb{R} \times M$.

In the sequel \cite{sequel} we construct a large class of examples of framed Hamiltonian manifolds for which this pseudoholomorphic curve exists. 
\end{rem}

We can say more about the invariant measure $\sigma$ from Theorem \ref{thm:invariantMeasures} in the case where the framed Hamiltonian structure is not only exact, but also of ``confoliation type''.

\begin{defn} \label{defn:confoliationType}
A framed Hamiltonian structure $\eta = (\lambda, \omega)$ on $M^{2n+1}$ is of \textbf{confoliation type} if $\omega$ is exact, and moreover admits a primitive $\nu$ such that
$$\nu \wedge \omega^n \geq 0.$$
\end{defn}

\begin{defn}
Let $(M^{2n+1}, \eta = (\lambda, \omega))$ be an exact framed Hamiltonian manifold of confoliation type. Denote by
$$\text{ConfPrim}(\omega)$$
the set of all of primitives $\nu$ of $\omega$ such that $\nu \wedge \omega^n \geq 0$. 
\end{defn}

We now state our main structure theorem for the invariant measures produced by Theorem \ref{thm:invariantMeasures}, in the case where the framed Hamiltonian manifold is of confoliation type.

\begin{thm} \label{thm:confoliations}
Let $(M^{2n+1}, \eta = (\lambda, \omega))$ be an exact framed Hamiltonian manifold of confoliation type satisfying the conditions of Theorem \ref{thm:invariantMeasures}. 

Then at least one of the following statements hold:
\begin{itemize}
    \item The Hamiltonian vector field $X$ has a periodic orbit,
    \item The invariant measure $\sigma$ is supported in the closed set $$\bigcap_{\nu \in \text{ConfPrim}(\omega)} \{\nu(X) = 0\}$$
    consisting of the intersection of the zero sets of the smooth function $\nu(X)$, taken across all primitives $\nu \in \text{ConfPrim}(\omega)$. 
\end{itemize}
\end{thm}

\begin{rem}
The closed set
$$\bigcap_{\nu \in \text{ConfPrim}(\omega)} \{\nu(X) = 0\}$$
is itself invariant under the flow of the Hamiltonian vector field $X$. 

Let $\nu \in \text{ConfPrim}(\omega)$ be any primitive of $\omega$ such that $\nu \wedge \omega^n \geq 0$. Denote by $\phi_X^t$ the time $t$ flow of the Hamiltonian vector field $X$. Then, since
$$(\phi_X^t)^*\omega = \omega$$
for any $t$, it follows that $(\phi_X^t)^*\nu$ is also a primitive of $\omega$ for any $t$, by the fact that the exterior derivative commutes with pullback maps on differential forms. 

Moreover, for any $t$,
\begin{align*}
    (\phi_X^t)^*\nu \wedge \omega^n &= (\phi_X^t)^*(\nu \wedge \omega^n) \\
    &\geq 0.
\end{align*}

It follows that $(\phi_X^t)^*\nu \in \text{ConfPrim}(\omega)$ for any $t$. This, along with the fact that $X$ is invariant under its own flow, implies immediately that 
$$\bigcap_{\nu \in \text{ConfPrim}(\omega)} \{\nu(X) = 0\}$$
is invariant under the flow of $X$. 
\end{rem}

\begin{rem}
It is, as of now, unclear what the weakest possible conditions are that ensure that the invariant measure $\sigma$ from Theorem \ref{thm:invariantMeasures} is not fully supported on the manifold $M$.
\end{rem}

\begin{rem} Theorem \ref{thm:confoliations} can be compared to Theorem $1.2$ of \cite{TaubesVectorFields}. The latter works in dimension three, while the former works in any odd dimension, albeit with the requirement that the conditions of Theorem \ref{thm:invariantMeasures} are satisfied.
\end{rem}

We also note that a weaker version of Corollary \ref{cor:uniqueErgodicity} in the non-exact case can be proved using a classical result of Schwartzman \cite{Schwartzman57} and the results of Hutchings-Taubes \cite{HutchingsTaubes09}. The results in \cite{HutchingsTaubes09} make use of embedded contact homology, a Floer-theoretic invariant for contact manifolds defined using certain counts of embedded pseudoholomorphic curves. The result is formulated and proved in Appendix \ref{sec:weakerNonExact}.
 
\subsection{Existence of $\omega$-finite pseudoholomorphic curves}

As noted briefly in Remark \ref{rem:examples}, the results of this paper apply to framed Hamiltonian manifolds $(M, \eta = (\lambda, \omega))$ which satisfy the conditions of Theorem \ref{thm:invariantMeasures}. That is, 
there is an $\eta$-adapted cylinder $(\mathbb{R} \times M, J, g)$ and an $\omega$-finite pseudoholomorphic curve
$$\mathbf{u} = (u, C, j, \mathbb{R} \times M, J, D, \mu).$$

At first glance, this requirement seems to complicate matters significantly. However, there are many rich pseudoholomorphic curve-based invariants for symplectic manifolds, which we leverage in \cite{sequel} to construct a large class of framed Hamiltonian manifolds satisfying the conditions of Theorem \ref{thm:invariantMeasures}. 

Recall from Example \ref{exe:hypersurfaces} that any closed, regular energy level $M = H^{-1}(0)$ of a Hamiltonian function $H$ on a closed symplectic manifold $(W, \Omega)$ admits a  framed Hamiltonian structure $\eta = (\lambda, \omega)$ such that $\omega = \Omega|_M$. The Hamiltonian structure $\eta$ also has the property that the associated Hamiltonian vector field $X$ coincides with the Hamiltonian vector field $X_H$ induced on $M$ by $H$ and $\Omega$. 

Our construction makes use of the \emph{Gromov-Witten invariants} of $W$, which were introduced in Gromov's original paper \cite{Gromov85} which initiated the study of pseudoholomorphic curves in symplectic geometry. Roughly, the Gromov-Witten invariants of $W$ are counts of pseudoholomorphic curves in $W$, satisfying various conditions on their topology and incidence with submanifolds of $W$. If $W$ has sufficiently rich Gromov-Witten invariants, we are able to deduce that there exists an integer $G \geq 0$ such that for any tame almost-complex structure $J$, there is a pseudoholomorphic curve
$$\mathbf{u} = (u, C, j, W, J, D, \mu)$$
crossing the hypersurface $M$. Furthermore, the domain $C$ has genus bounded above by $G$. 

Recall an almost-complex structure $J$ is \emph{tame} if $\Omega(V, JV) > 0$ for any tangent vector $V$.

One class of symplectic manifolds with sufficiently rich Gromov-Witten invariants, which are featured in the applications of our results in the subsequent subsection, are ``$\text{GW}_G$-connected'' symplectic manifolds. 

\begin{defn} \label{defn:GWconnected}
Fix an integer $G \geq 0$. A closed symplectic manifold $(W, \Omega)$ is \textbf{$\text{GW}_G$-connected} if there is a homology class $A \in H_2(W; \mathbb{Z})$ and integer $m \geq 0$ such that the Gromov-Witten map 
$$\Psi^W_{A, G, m+2}(e, e; -): H^*(W^m; \mathbb{Q}) \otimes H^*(\overline{\mathcal{M}}_{g, m+2}; \mathbb{Q}) \to \mathbb{Q}$$
does not vanish, where $e$ is the point class. 
\end{defn}

See \cite[Section $3$]{sequel} for an explanation of Gromov-Witten invariants and the notation used in Definition \ref{defn:GWconnected}. It is not required for any of the arguments in this paper. 

It is sufficient at this point to present an informal description of the notion of $\text{GW}_G$-connectedness of a symplectic manifold $(W, \Omega)$. Roughly, if $(W, \Omega)$ is $\text{GW}_G$-connected, we are able to deduce the following. For any tame almost-complex structure $J$ and any pair of points $p, q \in W$, there is a pseudoholomorphic curve
$$\mathbf{u}_{p,q} = (u_{p,q}, C, j, W, J, D, \mu)$$
with $(C, j)$ a closed, possibly disconnected Riemann surface of genus bounded above by $G$ and the image $u_{p,q}(C)$ passing through $p$ and $q$. 

\begin{exe}
\label{exe:CPNisConnected} In Gromov's seminal paper \cite{Gromov85}, it is proved in Theorem $0.2.B$ that $\mathbb{CP}^n$ with the standard symplectic structure is $\text{GW}_0$-connected (this is also sometimes called ``rationally connected'' in the symplectic geometry literature). 
\end{exe}

The following proposition is proved in \cite{sequel}. 

\begin{prop}
\label{prop:gwGConnected} \cite[Proposition $1.7$]{sequel} Fix an integer $G \geq 0$. Let $M = H^{-1}(0)$ be a closed, regular energy level of a smooth Hamiltonian function $H$ on a $\text{GW}_G$-connected symplectic manifold $(W, \Omega)$. Then there is an $\eta$-adapted cylinder $(\mathbb{R} \times M, J, g)$ and a feral pseudoholomorphic curve
$$\mathbf{u} = (u, C, j, \mathbb{R} \times M, J, D, \mu).$$
\end{prop}

We note that Proposition \ref{prop:gwGConnected} shows that there is a \emph{feral} pseudoholomorphic curve in an $\eta$-adapted cylinder over any regular energy level in a $\text{GW}_G$-connected symplectic manifold $(W, \Omega)$. Feral pseudoholomorphic curves are always $\omega$-finite, so this is stronger than what we require for Theorem \ref{thm:invariantMeasures}.

Now we present an informal description of the proof of Proposition \ref{prop:gwGConnected}. This proof works with weaker assumptions regarding the Gromov-Witten theory of $W$ than $\text{GW}_G$-connectedness. The exact required condition can be found in \cite[Section $3$]{sequel} and the statement of \cite[Theorem $1.3$]{sequel}.

Suppose that $M = H^{-1}(0)$ is a closed, regular energy level of a smooth Hamiltonian function $H$ on a closed symplectic manifold $(W, \Omega)$. Suppose further that $(W, \Omega)$ has sufficiently rich Gromov-Witten theory, so that there are integers $G \geq 0$ and $m \geq 0$ such that for any tame almost-complex structure $J$, there is a pseudoholomorphic curve
$$\mathbf{u} = (u, C, j, W, J, D, \mu)$$
such that $C$ has genus bounded above by $G$, $\mu$ has at most $m+2$ points, and the image $u(C)$ crosses the hypersurface $M$. 

Starting with a compatible almost-complex structure $J_0$, we degenerate it near the hypersurface $M$ as part of a ``neck-stretching procedure'', producing a one-parameter family of tame almost-complex structures $\{J_L\}_{L \in [0,\infty)}$. For each $L \in [0, \infty)$, Gromov-Witten theory guarantees the existence of a pseudoholomorphic curve
$$\mathbf{u}_L = (u_L, C_L, j_L, W, J_L, D_L, \mu_L).$$

Taking a certain limit of the curves $\mathbf{u}_L$ as $L \to \infty$ produces a feral pseudoholomorphic curve in an $\eta$-adapted cylinder over $M$ as desired. 

While neck stretching techniques are ubiquitous in symplectic geometry, they almost exclusively require $M$ to be a special kind of hypersurface known as a ``contact-type'' hypersurface or a ``stable Hamiltonian'' hypersurface. In our setting, $M$ is not necessarily contact-type or stable Hamiltonian, which presents several technical difficulties. 

First, the sequence of curves $\mathbf{u}_L$ do not necessarily satisfy local area bounds as they would in the case where $M$ is contact-type or stable Hamiltonian (this is referred to as ``bounded Hofer energy'' in the symplectic literature). Namely, for any compact set $K \subset W$, the area of the image of the curve $\mathbf{u}_L$ inside of $K$ may diverge to infinity. 

Second, the number of connected components of the domain $C_L$ for any $L \in [0,\infty)$ may possibly increase without bound. Without some a priori control on the number of connected components of $C_L$, we would be unable to extract our desired feral curve from the family $\mathbf{u}_L$. 

Third, we are only able to guarantee a weak form of convergence, called ``exhaustive Gromov convergence'' (see \cite{FishHoferExhaustive}) of the family $\mathbf{u}_L$ to a pseudoholomorphic curve 
$$\mathbf{u} = (u, C, j, \mathbb{R} \times M, J, D, \mu).$$

We can show without much effort that $\mathbf{u}$ has finite genus and a finite number of marked and nodal points. However, it is not clear from the definition of exhaustive Gromov convergence that the number $\text{Punct}(C)$ of generalized punctures of $C$ is finite, which is required for $\mathbf{u}$ to be feral rather than merely $\omega$-finite. 

The first difficulty is addressed using new analytical results from Fish-Hofer's work \cite{FishHoferFeral} on feral pseudoholomorphic curves. 

The second difficulty requires additional arguments, which show the desired a priori bounds on the number of connected components of $C_L$.

The third difficulty requires delicate arguments like those in the proof of Proposition $4.49$ in \cite{FishHoferFeral}. 

Given these difficulties, the construction is lengthy and technical. Moreover, the rigorous development of the neck stretching technique along arbitrary hypersurfaces in symplectic manifolds may be of independent interest. It is for these reasons that we have deferred the construction to the sequel \cite{sequel}. 

\subsection{Applications}

The combination of Theorem \ref{thm:mainNonExact}, Theorem \ref{thm:mainExact} and Proposition \ref{prop:gwGConnected} give the following broad result regarding the question of unique ergodicity for Hamiltonian vector fields on closed, regular energy levels in closed, $\text{GW}_G$-connected symplectic manifolds. Another application to the question of non-unique ergodicity in symplectic $4$-manifolds, using Seiberg-Witten theory to construct $\omega$-finite pseudoholomorphic curves, is discussed in \cite[Proposition $1.7$]{sequel}.

\begin{prop} \label{prop:mainApplication}
Let $(W, \Omega)$ be a closed symplectic manifold of dimension $2n+2$ that is $\text{GW}_G$-connected for some integer $G \geq 0$. Let 
$$H: W \to \mathbb{R}$$
be any smooth Hamiltonian such that $0$ is a regular value, and set $M = H^{-1}(0)$.

Suppose that one of the following two conditions hold:
\begin{itemize}
    \item The restriction of the two-form $\Omega^n$ is not exact on $M$. 
    \item $\Omega$ is exact on one of the two components of $W \setminus M$.
\end{itemize} 

Then the Hamiltonian vector field $X_H$ on $M$ is not uniquely ergodic. 
\end{prop}

\begin{proof}
Let $\eta = (\lambda, \omega)$ denote the framed Hamiltonian structure on $M$ which is constructed as in \ref{exe:hypersurfaces}. Since $(W, \Omega)$ is $\text{GW}_G$-connected for some integer $G \geq 0$, $M$ satisfies the conditions of Theorem \ref{thm:invariantMeasures}. It remains to verify that $M$ satisfies either the conditions of Theorem \ref{thm:mainNonExact} or Theorem \ref{thm:mainExact}. Then the proof is complete by Corollary \ref{cor:uniqueErgodicity}. 

If $\omega^n$ is not exact, then $M$ satisfies the conditions of Theorem \ref{thm:mainNonExact}. 

Now consider the case where $\omega$ is exact. Let $\nu$ be any primitive of $\omega$. To show that $M$ satisfies the conditions of Theorem \ref{thm:mainExact}, we must show that
$$\text{Lk}(\omega) = \int_M \nu \wedge \omega$$
is not zero. 

Observe that $M$ separates $W$ into two components
$$W_+ = H^{-1}([0, \infty))$$
and
$$W_- = H^{-1}((-\infty, 0]).$$

The assumptions of the proposition allow us to assume without loss of generality that $\Omega$ is exact on $W_-$. 

By Stokes' theorem and the fact that $\omega$ is the restriction of the symplectic form $\Omega$ to $M$,
$$\text{Lk}(\omega) = \int_{W_-} \Omega \wedge \Omega > 0.$$

We conclude that if $\omega$ is exact, $M$ satisfies the conditions of Theorem \ref{thm:mainExact} as desired, and therefore the proof is complete. 
\end{proof}

\begin{rem}
In the case where $(W, \Omega)$ is four-dimensional and $\text{GW}_G$-connected for some $G \geq 0$, the result in Proposition \ref{prop:mainApplication} shows that only possible Hamiltonian energy levels $M = H^{-1}(0)$ on which the Hamiltonian vector field could be uniquely ergodic are those on which the restriction of $\Omega$ to $M$ is exact, but $\Omega$ is not exact on either component of $W \setminus M$. It would be interesting to construct an example of such an energy level with uniquely ergodic Hamiltonian vector field, which would show that Proposition \ref{prop:mainApplication} is sharp in four dimensions. 
\end{rem}

\begin{rem}
Proposition \ref{prop:mainApplication} requires Theorem \ref{thm:mainNonExact}, which concerns existence of interesting invariant measures on non-exact framed Hamiltonian manifolds. This is because an arbitrary Hamiltonian energy level $M = H^{-1}(0)$ in a closed symplectic manifold $(W, \Omega)$ may not be such that the restriction of $\Omega$ to $M$ is exact. 

On the other hand, the following Corollary \ref{cor:energySurfaces} only requires Theorem \ref{thm:mainExact}, which concerns existence of interesting invariant measures on exact framed Hamiltonian manifolds and is closely related to results of \cite{GinzburgNiche2015} and \cite{TaubesVectorFields}. 
\end{rem}

We can also recover the following result of \cite{GinzburgNiche2015}, using an argument identical to the one in the proof of Proposition \ref{prop:mainApplication}. 

\begin{cor} \label{cor:energySurfaces}
\cite[Corollary 1.2]{GinzburgNiche2015} Let $M = H^{-1}(0)$ be the zero set of a Hamiltonian function $H$ on $\mathbb{R}^{2n+2}$ for which $0$ is not a critical value. Then the Hamiltonian vector field $X_H$ associated to $H$ and the standard symplectic form on $\mathbb{R}^{2n+2}$ is not uniquely ergodic. 
\end{cor}

\begin{proof}
Let $M = H^{-1}(0)$ be a closed, regular energy level of a smooth Hamiltonian function $H$ on $\mathbb{R}^{2n+2}$. 

Via a rescaling procedure like in the proof of Theorem $1$ in \cite{FishHoferFeral}, we can assume that there is an open neighborhood of the origin $U$ in $\mathbb{R}^{2n+2}$ containing $M$ and a symplectic embedding 
$$\iota: (U, \Omega_0) \hookrightarrow (\mathbb{CP}^n, \Omega_{\text{FS}}).$$

Here $\Omega_0$ denotes the standard symplectic form on $\mathbb{R}^{2n+2}$, and $\Omega_{\text{FS}}$ denotes the standard Fubini-Study symplectic form on $\mathbb{CP}^n$. 

Define a smooth Hamiltonian function $\overline{H}$ on $\mathbb{CP}^n$ that agrees with $H \circ \iota^{-1}$ in a neighborhood of $\iota(M)$ with closure contained in $\iota(U)$. 
Therefore, we find $\iota(M) = \overline{H}^{-1}(0)$ and the Hamiltonian vector field $X_{\overline{H}}$ on $\iota(M)$ coincides with the pushforward $\iota_*X_H$ of the Hamiltonian vector field $X_H$ on $M$. It follows that if $X_{\overline{H}}$ is not uniquely ergodic, then $X_H$ is not uniquely ergodic. 

Recall from Example \ref{exe:CPNisConnected} that $\mathbb{CP}^n$ is $\text{GW}_0$-connected. The restriction of $\Omega_{\text{FS}}$ to $\iota(U)$ is exact. Note that $\iota(U)$ by definition contains one of the two components of $\mathbb{CP}^2 \setminus \iota(M)$. 

By Proposition \ref{prop:gwGConnected}, the symplectic manifold $(\mathbb{CP}^2, \Omega_{\text{FS}})$ and the hypersurface $\iota(M) = \overline{H}^{-1}(0)$ satisfy the conditions of Proposition \ref{prop:mainApplication}. The result follows from Proposition \ref{prop:mainApplication}.
\end{proof}

\begin{rem}
Theorem $1.1$ of \cite{TaubesVectorFields} shows that, for any exact, $3$-dimensional framed Hamiltonian manifold $M$, as long as the obstruction $\text{Lk}(\omega)$ does not vanish, then the Hamiltonian vector field is not uniquely ergodic. 

It is expected that Seiberg-Witten theory will allow one to construct an $\omega$-finite pseudoholomorphic curve in $\mathbb{R} \times M$, showing that any exact, $3$-dimensional framed Hamiltonian manifold satisfies the conditions of Theorem \ref{thm:invariantMeasures}. Then Corollary \ref{cor:uniqueErgodicity} would imply the main theorem of \cite{TaubesVectorFields} for all exact framed Hamiltonian manifolds. This construction is work in progress by the author.
\end{rem}

\subsection{Background on pseudoholomorphic curves and Hamiltonian dynamics} \label{subsec:history}

In this subsection, we give a quick overview of applications of pseudoholomorphic curves to problems in Hamiltonian dynamics, and place our present work in the context of this story. 

The work most relevant to our setting is Hofer's work \cite{HoferWeinstein}. Given a contact manifold $M$ equipped with a contact form and associated Reeb vector field $X$, the manifold $\mathbb{R} \times M$ has a natural symplectic structure, and moreover can be equipped with a translation-invariant almost-complex structure $J$ that incorporates the vector field $X$. Hofer showed that, in this setting, $J$-holomorphic curves satisfying a certain finite energy condition limit to periodic orbits of the Reeb vector field. 

One way to describe this finite energy condition is as follows. Let $u: C \to \mathbb{R} \times M$ be a proper $J$-holomorphic map in $\mathbb{R} \times M$, where $C$ is some non-compact Riemann surface equipped with an almost-complex structure which we suppress. Then for any $t \in \mathbb{R}$ outside of a set of Lebesgue measure zero, the intersection
$$C \cap \{t\} \times M$$
is a set of closed, embedded curves inside $S$, colloquially referred to as a ``level set'' of the curve $S$. Then Hofer's finite energy condition implies that there is some \emph{uniform bound} on the lengths of the level sets, where length is measured by the choice of a natural cylindrical metric on $\mathbb{R} \times M$ that is preserved by $J$. 

Then, if $u$ satisfies Hofer's finite energy condition, the following holds. For any generic increasing sequence $t_k \to \infty$ consider the sequence of level sets
$$C \cap \{t_k\} \times M.$$

Then, after possibly passing to a subsequence, there is a collection of periodic orbits $P \subset M$ such that the sets of loops $C \cap \{t_k\} \times M$ approach $P$ in the Hausdorff distance:

$$\lim_{k \to \infty}(\sup_{x \in  \cap \{t_k\} \times M} \text{dist}(x, P) + \sup_{y \in P} \text{dist}(y, C \cap \{t_k\} \times M)) = 0.$$

We illustrate this convergence in Figure \ref{fig:finiteEnergyCurve} below. 

\begin{figure}[ht]  \includegraphics[scale=.25]{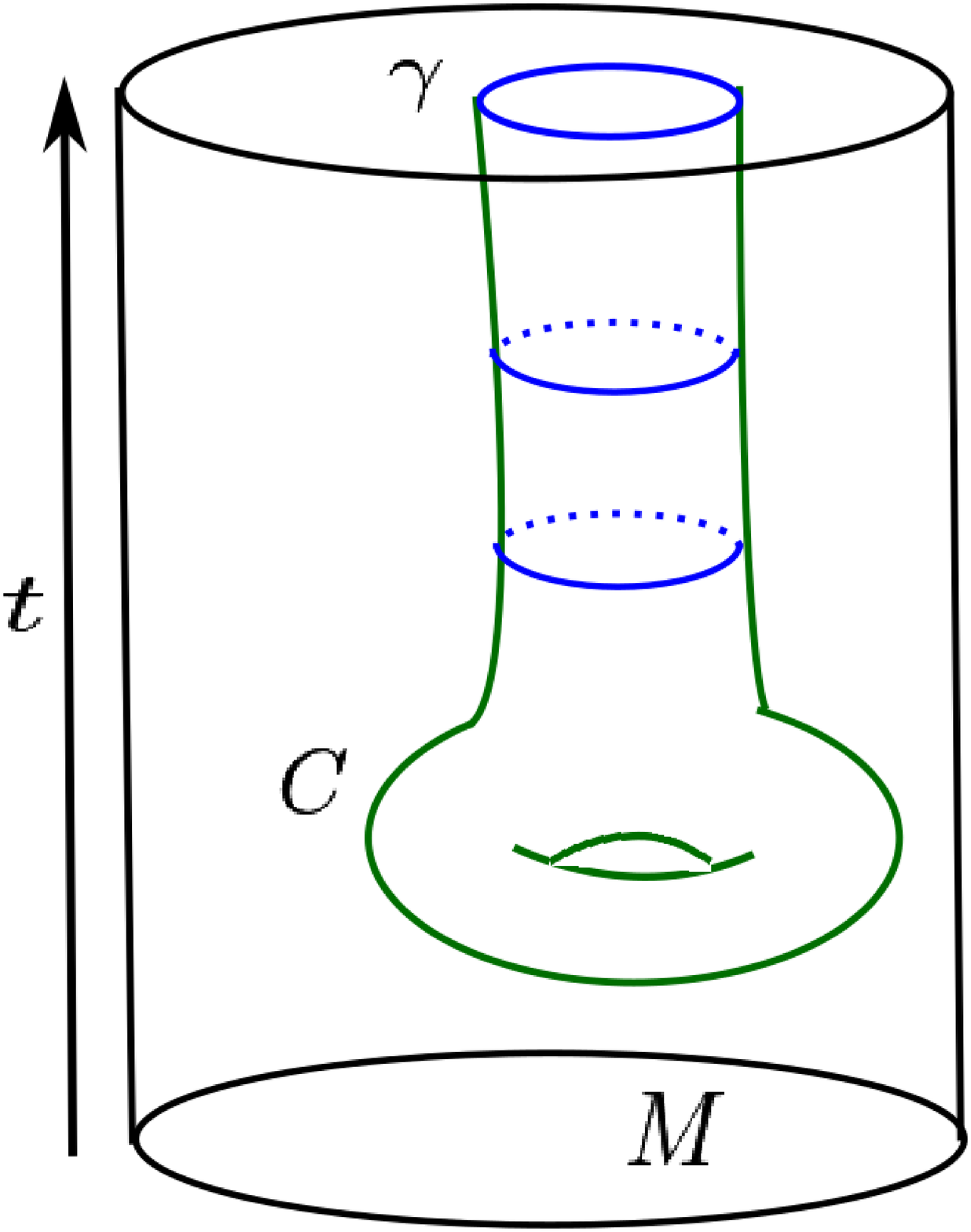} \caption{\label{fig:finiteEnergyCurve} The level loops (blue) of the curve $C$ (green) converge to the orbit $\gamma$ as the $t$-coordinate goes to infinity.} \end{figure}

This asymptotic result showed that, in order to detect a periodic orbit of the Reeb vector field, one needs only to prove the existence of such a pseudoholomorphic curve. At first glance, it is not clear that this makes the problem any easier.

However, Hofer was in the same work able to use, in certain situations, local existence results for holomorphic disks and a bubbling off argument to produce such curves. Concretely, this resulted in a proof of the Weinstein conjecture in dimension three for contact manifolds that are overtwisted, have non-zero second homotopy group, or diffeomorphic to $S^3$. 

However, until only recently, applying pseudoholomorphic curve techniques to Hamiltonian flows more general than Reeb flows (or stable Hamiltonian flows) seemed untenable. This is because, outside of the contact/stable Hamiltonian setting, standard techniques used to produce pseudoholomorphic curves do not necessarily produce curves satisfying Hofer's finite energy condition. 

Such curves can a priori have rather wild level sets that have unbounded length, and as a result are not necessarily guaranteed to converge to a periodic orbit, as illustrated in Figure \ref{fig:infiniteEnergyCurve}. 

\begin{figure}[ht]  \includegraphics[scale=.25]{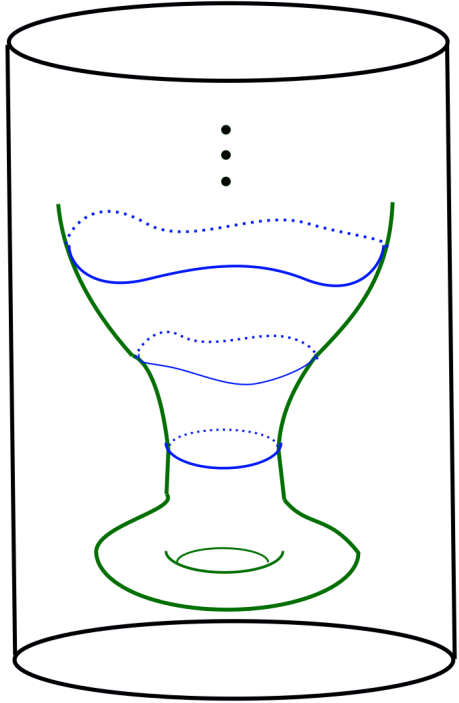} \caption{\label{fig:infiniteEnergyCurve} An infinite energy $J$-holomorphic curve (green). The level sets are shown in blue, analogously to Figure \ref{fig:finiteEnergyCurve}}. \end{figure}

Moreover, the analysis of pseudoholomorphic curves that do not satisfy Hofer's finite energy condition was considered unmanageable (see the introduction of \cite{FishHoferFeral} for an exhaustive discussion of issues in this setting).

However, recent work of Fish-Hofer \cite{FishHoferFeral} has shown that it is still possible to prove interesting results about the dynamics of Hamiltonian vector fields using a new class of infinite energy pseudoholomorphic curves they term ``feral curves''. While these pseudoholomorphic curves may not limit to periodic orbits, they do limit in a certain sense to \emph{closed invariant subsets} of the vector field that may be more general than periodic orbits. 

The pseudoholomorphic curves we use in our work are more general than feral curves. The invariant measures we construct can be regarded as ``limit sets'' of our pseudoholomorphic curves. 

Qualitatively, this is done as follows. Choose again a sequence of level sets
$$C \cap \{t_k\} \times M$$
with $t_k \to \infty$ as before. The $\omega$-finiteness of $C$ implies that, in a certain weak sense, the tangent lines of these level sets are asymptotic to the line bundle spanned by the Hamiltonian vector field $X$ as $k \to \infty$.

Now note that any set of closed, one-dimensional loops in a manifold equipped with a measure can be considered as Radon measures, merely by defining the Radon measure by integration over this set of loops. Therefore, any level set of a $\omega$-finite curve can be thought of as a Radon measure. If we take a sequence of level sets
$$C \cap \{t_k\} \times M$$
with $t_k \to \infty$ as before, we obtain a sequence of Radon measures on $M$. 

However, the norms of these Radon measures could a priori unbounded since our curve does not have finite Hofer energy. To remedy this, we re-normalize all of the Radon measures to get a sequence $\sigma_k$ of Radon measures of norm $1$. Then the compactness of the weak unit ball of the space of Radon measures (see Theorem \ref{thm:radonMeasureCompactness}) shows that we can extract a weak limit $\sigma$. Moreover, the fact that the tangent lines of the level sets approach the Hamiltonian vector field allows us to show that $\sigma$ is actually invariant.

\subsection{Roadmap}

The rest of the paper is organized as follows.

In Section \ref{sec:proofs1} we prove Theorem \ref{thm:invariantMeasures}. We begin in Section \ref{subsec:boundedLambdaEnergy} by proving that existence of an $\omega$-finite curve
$$\mathbf{u} = (u, C, j, \mathbb{R} \times M, J, D, \mu)$$
in an $\eta$-adapted cylinder over a framed Hamiltonian manifold $(M, \eta)$ with \emph{bounded $\lambda$-energy} implies the existence of a periodic orbit of the Hamiltonian vector field $X$. The proof of this statement is straightforward in the case where the domain has finite genus and a finite number of ends, but the proof in the case where we do not make these assumptions requires a lengthy digression into the compactness and regularity theory of $J$-holomorphic currents in almost-complex manifolds. In Section \ref{subsec:unboundedLambdaEnergy}, we complete the proof of Theorem \ref{thm:invariantMeasures} by producing an invariant probability measure from the level sets of an $\omega$-finite curve
$$\mathbf{u} = (u, C, j, \mathbb{R} \times M, J, D, \mu).$$

 In Section \ref{sec:proofs2}, we prove Theorems \ref{thm:mainNonExact}, \ref{thm:mainExact}, and \ref{thm:confoliations} using Theorem \ref{thm:invariantMeasures}. 
 
 In Appendix \ref{sec:measureTheory}, we write down supporting measure-theoretic definitions and results. 
 
 In Appendix \ref{sec:currents}, we discuss the theory of pseudoholomorphic currents in almost-Hermitian manifolds and use it to prove Proposition \ref{prop:periodicOrbits}, which is an essential part of the proof of Theorem \ref{thm:invariantMeasures}. 
 
 In Appendix \ref{sec:weakerNonExact} we prove a weaker version of Corollary \ref{cor:uniqueErgodicity} using alternate pseudoholomorphic curve-based techniques.

\textbf{Acknowledgements:}  I would like to thank my advisor, Helmut Hofer, for his encouragement, discussions regarding the work, and numerous suggestions which greatly clarified the exposition. I would also like to thank Joel Fish for enlightening discussions regarding pseudoholomorphic curves, and Clifford Taubes for originally pointing me to his paper on uniquely ergodic vector fields. I would also like to thank Camillo De Lellis and Vikram Giri for answering some questions I had regarding geometric measure theory and the regularity of semicalibrated currents. 

\section{Proof of Theorem \ref{thm:invariantMeasures}} \label{sec:proofs1}

Fix a framed Hamiltonian manifold $(M^{2n+1}, \eta= (\lambda, \omega))$ and an $\eta$-adapted cylinder $(\mathbb{R} \times M, J, g)$ as in Definition \ref{defn:adaptedCylinder}. 

Let 
$$\mathbf{u} = (u, C, j, \mathbb{R} \times M, J, D, \mu)$$
be an $\omega$-finite curve, whose existence is assumed in the statement of Theorem \ref{thm:invariantMeasures}. 

\subsection{Bounded $\lambda$-energy implies existence of a periodic orbit} \label{subsec:boundedLambdaEnergy} First, we observe that either $X$ has a periodic orbit or the $\lambda$-energy of $u$ is not uniformly bounded on any unbounded sequence $t_k$ of regular values of $a \circ u$.

This is stated as the following proposition.

\begin{prop} \label{prop:periodicOrbits}
Suppose that there is an unbounded sequence $t_k$ of real numbers in $\mathcal{R}$ such that $u(C) \cap \{t_k\} \times M$ is nonempty and $E_\lambda(u, t_k)$ is uniformly bounded in $k$. Then $X$ has a periodic orbit.
\end{prop}

Proposition \ref{prop:periodicOrbits} is straightforward in the case where the domain $C$ has finite genus, a finite number of ends, and a finite number of marked and nodal points. We assume none of these conditions, so there are additional complications and we must appeal to the theory of $J$-holomorphic currents to prove Proposition \ref{prop:periodicOrbits}. 

The proof may be of independent interest, but it requires a lengthy digression into the regularity and compactness theory of $J$-holomorphic currents in almost-Hermitian manifolds. Therefore, we direct the reader to Appendix \ref{sec:currents} for the proof. 

\subsection{Unbounded $\lambda$-energy permits construction of an invariant measure} \label{subsec:unboundedLambdaEnergy} 

Given Proposition \ref{prop:periodicOrbits}, we can proceed with the second part of the proof of Theorem \ref{thm:invariantMeasures} and construct the invariant measure $\sigma$ in the case where the $\lambda$-energy of $u$ is unbounded. 

In the subsequent arguments, we will often be integrating over level loops of the form $(a \circ u)^{-1}(t) \subset C$, for $t$ a regular value of the smooth function $a \circ u$. For any loop in $C$, we also introduce the notation $\mu^1_{u^*g}$ to denote the one-dimensional Hausdorff measure on the loop given by the volume form of the restriction of $u^*g$ to that loop. 

The following lemma is useful in the analysis below. Write $\nabla(a \circ u)$ for the gradient of the function $a \circ u$ with respect to the pullback metric $u^*g$, wherever this is well-defined.

Recall $\mathcal{R}$ is the set of regular values of the real-valued map $a \circ u$.

\begin{lem}[Non-negativity of $\lambda$] \label{lem:nonNegativityLambda}
 Let $\rho = \nabla(a \circ u)/|\nabla(a \circ u)|_{u^*g}$ wherever this is well-defined and $\tau = j\rho$. Then $u^*\lambda(\tau) \geq 0$. Moreover, the one-form $u^*\lambda$ is non-negative on the tangent lines of the loops $(a \circ u)^{-1}(t)$ for any $t \in \mathcal{R}$.
\end{lem}

\begin{proof}
Since $u$ is $J$-holomorphic,
\begin{align}
    u^*\lambda(\tau) &= \lambda(du(\tau)) \nonumber\\
    &= \lambda(du(j\rho)) \nonumber \\
    &= \lambda(J(du(\rho))) \tag{\textbullet} \label{eq:nonNegativityLambda1}\\
    &= da(du(\rho)) \tag{\textbullet\textbullet} \label{eq:nonNegativityLambda2}\\
    &= d(a \circ u)(\rho) \nonumber \\
    &= |\nabla(a \circ u)|_{u^*g} \nonumber\\
    &\geq 0. \nonumber
\end{align}

The equality (\ref{eq:nonNegativityLambda1}) follows from the Cauchy-Riemann equation
$$J \circ du = du \circ j$$
and the equality (\ref{eq:nonNegativityLambda2}) follows from the fact that $da = \lambda \circ J$, which is a direct consequence of the fact that $J$ is $\eta$-adapted (see Definition \ref{defn:adaptedJ}). 

The second conclusion of the lemma follows from the fact that, owing to the orientation conventions, one has that $u^*\lambda$ is, as a measure, a non-negative multiple of the measure
$$u^*\lambda(\tau)\mu^1_{u^*g}$$
on any regular level loop $(a \circ u)^{-1}(t)$. 
\end{proof}

Since $J$ is $\eta$-adapted, the tangent plane in $\mathbb{R} \times M$ spanned by $\partial_a$ and $X$ is preserved by $J$. Moreover (see Lemma \ref{lem:omegaNonnegative}), the two-form $\omega$ is non-negative on any tangent plane preserved by $J$, and zero exactly on the tangent plane spanned by $\partial_a$ and $X$. 

Furthermore, recall that the $\omega$-energy
$$E_\omega(u) = \int_C u^*\omega$$
is finite. It follows that asymptotically, the tangent planes of $u(C)$ should approach the span of $\partial_a$ and $X$, since any tangent plane which is not equal to this span ``consumes'' some $\omega$-energy.

This is made rigorous in the following weak-type estimate, proved by Fish and Hofer in \cite{FishHoferFeral}. The proof of this lemma makes no assumptions about the pseudoholomorphic curve $\mathbf{u}$ beyond the fact that the pseudoholomorphic map $u$ is proper and that $E_\omega(u)$ is finite.

\begin{lem} \cite[Lemma $4.27$]{FishHoferFeral} \label{lem:almostVertical}
Pick constants $0 < \theta < 1$, $\epsilon > 0$. Then, the Lebesgue measure of the set of real numbers
$$\mathcal{S}_{\theta, \epsilon} = \{t \in \mathcal{R}\,|\,\mu^1_{u^*g}(\{p \in (a \circ u)^{-1}(t)\,|\,|u^*\lambda|_{u^*g} < \theta\}) > \epsilon\}$$
is bounded above by 
$$(\epsilon(1-\theta^2))^{-1}E_\omega(u) < \infty.$$
\end{lem}

This lemma is used for the following crucial technical estimate. 

\begin{lem} \label{lem:niceOneForms}
Pick two sequences of real numbers $\theta_k \to 1$ and $\epsilon_k \to 0$. Pick an additional sequence $t_k \to \infty$ such that the following holds for every $k$:
\begin{enumerate}
    \item $t_k \in \mathcal{R}$,
    \item $(a \circ u)^{-1}(t_k)$ does not intersect the boundary $\partial C$ of $C$,
    \item For every $k$, $t_k \not\in \mathcal{S}_{\theta_k, \epsilon_k}$, where $\mathcal{S}_{\theta_k, \epsilon_k}$ is the set defined in Lemma \ref{lem:almostVertical}.
\end{enumerate}

Then for any one-form $\beta$ on $M$ such that
$$\beta(X) \equiv 0,$$
there is a sequence of constants $c_k$ decaying to zero as $k \to \infty$ depending only on $\theta_k$ and $\epsilon_k$ such that for every $k$,
$$|\int_{(a \circ u)^{-1}(t_k)} u^*\beta| \leq c_k(E_\lambda(u, t_k) + 1)\|\beta\|_g.$$
\end{lem}

\begin{proof}
Note that if $W$ is some tangent vector on $M$, then 
$$|\lambda(W)|^2 + \|\pi_\xi W\|^2_g = \|W\|^2_g.$$
 
Therefore, if $\lambda$ has norm bounded above by $\theta$ on the line spanned by $W$, it follows that
$$\|\pi_\xi W\|_g \leq (1 - \theta)^{1/2}\|W\|_g.$$

This estimate is useful because since $\beta(X) \equiv 0$, we will have
\begin{align*}
|\beta(W)| &= |\beta(\pi_\xi(W))| \\
&\leq \|\pi_\xi(W)\|_g\|\beta\|_g \\
&\leq (1 - \theta)^{1/2}\|W\|_g\|\beta\|_g
\end{align*}

For every $k$, write
$$\Sigma_k = \{p \in (a \circ u)^{-1}(t_k)\,|\, \|u^*\lambda\|_{u^*g}(p) \leq \theta_k\}.$$

Then by non-negativity of $\lambda$ (Lemma \ref{lem:nonNegativityLambda}) and the definition of the set $\Sigma_k$, 
$$E_\lambda(u, t_k) \geq \int_{(a \circ u)^{-1}(t_k)\setminus \Sigma_k} u^*\lambda \geq \theta_k \mu^1_{u^*g}((a \circ u)^{-1}(t_k)\setminus \Sigma_k).$$

Now 
\begin{align*}
|\int_{(a \circ u)^{-1}(t_k)} u^*\beta| &\leq (\int_{(a \circ u)^{-1}(t_k)\setminus \Sigma_k} + \int_{\Sigma_k}) \|u^*\beta\|_{u^*g} \mu^1_{u^*g}(\zeta) \\
&\leq \mu^1_{u^*g}((a \circ u)^{-1}(t_k)\setminus \Sigma_k)(1 - \theta_k^2)^{1/2}\|\beta\|_g + \epsilon_k\|\beta\|_g \\
&\leq (\theta_k^{-1}(1 - \theta_k^2)^{1/2} + \epsilon_k)(E_\lambda(u, t_k) + 1)\|\beta\|_g \\
&\leq c_k(E_\lambda(u, t_k) + 1)\|\beta\|_g
\end{align*}
where $c_k \to 0$ as $k \to \infty$.
\end{proof}

Now we proceed with construction of an $X$-invariant measure. 

Pick any unbounded sequence of real numbers $\{t_k\}$ such that $u(C)$ intersects $\{t_k\} \times M$ for every $k$. Without loss of generality, after passing to a subsequence we may assume that $t_k$ either increases monotonically to $\infty$ or decreases monotonically to $-\infty$. We will perform the construction assuming the former holds, as the proof in the latter case is identical. 

Due to Proposition \ref{prop:periodicOrbits} and Lemma \ref{lem:nonNegativityLambda}, we may assume that
$$\lim_{k \to \infty} E_\lambda(u, t_k) = \infty.$$

Otherwise, $X$ has a periodic orbit.

Next, choose any sequence of real numbers $\theta_j \to 1$, and $\epsilon_j \to 0$. By Lemma \ref{lem:almostVertical}, the set $\mathcal{S}_{\theta_j, \epsilon_j}$ has finite measure for any $j$. Let $\{t_k\} \to \infty$ be the original monotonically increasing sequence fixed at the beginning of the proof. 

It follows that for any $j$, there exists some large $k(j)$ such that for any $k \geq k(j)$, the intersection of the closed interval $[t_{k(j)} - 1/j, t_{k(j)} + 1/j]$ with $\mathcal{S}_{\theta_j, \epsilon_j}$ does not have full measure equal to $2/j$. By Sard's theorem, there is some $\delta_{k(j)} \in (-1/j, 1/j)$ such that $t_{k(j)} + \delta_{k(j)}$ is a regular value of $a \circ u$, $u(C)$ intersects $\{t_{k(j)} + \delta_{k(j)}\} \times M$, and $t_{k(j)} + \delta_{k(j)} \not\in \mathcal{S}_{\theta_j, \epsilon_j}$. 

We conclude that after passing to a subsequence of the $t_k$ and re-labeling, there is a sequence $\delta_k \to 0$ such that $u(C)$ intersects $\{t_k + \delta_k\} \times M$ for every $k$ and that the sequence $\{t_k + \delta_k\}$ satisfies the conditions of Lemma \ref{lem:niceOneForms}. 

For the sake of brevity, we will overload our notation and re-label the number $t_k + \delta_k$ by $t_k$. Therefore, we may assume that $u(C)$ intersects $\{t_k\} \times M$ for every $k$, and that the sequence $\{t_k\}$ satisfies the conditions of Lemma \ref{lem:niceOneForms}. 

Given these considerations, Lemma \ref{lem:niceOneForms} has the following useful corollary.

\begin{cor} (One-forms annihilating $X$ disappear at infinity) \label{cor:betaVanishing}
Let $\beta$ be a one-form such that 
$$\beta(X) \equiv 0.$$

Then
$$\lim_{k \to \infty} E_\lambda(u, t_k)^{-1}\int_{(a \circ u)^{-1}(t_k)} u^*\beta = 0.$$
\end{cor}

\begin{proof}
From Lemma \ref{lem:niceOneForms}, we find that there is a sequence of constants $c_k \to 0$ such that 
$$|\int_{(a \circ u)^{-1}(t_k)} u^*\beta| \leq c_k(E_\lambda(u, t_k) + 1)\|\beta\|_g.$$

Therefore, 
$$\lim_{k \to \infty} E_\lambda(u, t_k)^{-1}|\int_{(a \circ u)^{-1}(t_k)} u^*\beta| \leq  \lim_{k \to \infty} c_k E_\lambda(u,t_k)^{-1}(E_\lambda(u, t_k) + 1)\|\beta\|_g = 0.$$
\end{proof}

This sequence produces an invariant measure in the following way. For every $k$ define a Radon measure $\sigma_k$ by $$\sigma_k(f) = E_\lambda(u, t_k)^{-1}\int_{(a \circ u)^{-1}(t_k)} u^*(f\lambda)$$
for any continuous function $f$ on $M$. 

The norm of an arbitrary Radon measure $\sigma$ is its norm as an element of $C^0(M)^*$:
$$\|\sigma\| = \sup_{f \in C^0(M)} (\|f\|_{C^0(M)})^{-1}|\mu(f)|.$$

It is clear from the non-negativity of $\lambda$ that for any $k$, 
\begin{align*}
    \|\sigma\| &= \sup_{f \in C^0(M) \setminus\{0\}} (\|f\|_{C^0(M)})^{-1}|\sigma_k(f)| \\
    &= \sup_{f \in C^0(M) \setminus\{0\}} (\|f\|_{C^0(M)})^{-1}E_{\lambda}(u, t_k)^{-1} \int_{(a \circ u)^{-1}(t_k)} u^*f\lambda \\
    &\leq \sup_{f \in C^0(M) \setminus\{0\}} (\|f\|_{C^0(M)})^{-1}E_{\lambda}(u, t_k)^{-1} (\|f\|_{C^0(M)}\int_{(a \circ u)^{-1}(t_k)} u^*\lambda) \\
    &\leq \sup_{f \in C^0(M) \setminus\{0\}} E_{\lambda}(u, t_k)^{-1} \int_{(a \circ u)^{-1}(t_k)} u^*\lambda \\
    &= 1.
\end{align*}

Recall that the space of Radon measures with uniform norm bound is compact in the weak topology of linear functionals on $C^0(M)$ (see Theorem \ref{thm:radonMeasureCompactness}). Therefore, after passing to a subsequence, there is a limiting Radon measure $\sigma$, defined by the identity
$$\sigma(f) = \lim_{k \to \infty} \sigma_k(f)$$
for any continuous function $f$ on $M$. 

We now show that $\sigma$ is an invariant probability measure. First, $\sigma$ is certainly a probability measure, since $\sigma_k(1) = 1$ for every $k$. 

Next, recall it suffices by Lemma \ref{lem:linearizedInvariance} to show that $\sigma$ is equal to zero on any function of the form $\nabla_X f$ for any $C^1$ function $f$ on $M$. 

By definition,
\begin{align*}
    \nabla_X f \cdot \lambda &= df(X) \cdot \lambda \\
    &= (df \wedge \lambda)(X, -) + \lambda(X) \cdot df \\
    &= (df \wedge \lambda)(X, -) + df.
\end{align*}

Therefore by Stokes' theorem, 
\begin{align*}
    \sigma_k(\nabla_X f) &= E_\lambda(u, t_k)^{-1}\int_{(a \circ u)^{-1}(t_k)} u^*((df\wedge\lambda)(X, -)).
\end{align*}

Observe that $(df \wedge \lambda)(X, -)$ is a one-form that evaluates to $0$ when $X$ is plugged in. 

It follows that $\sigma$ is an invariant measure immediately by applying Corollary \ref{cor:betaVanishing} with $\beta = (df \wedge \lambda)(X, -)$, and so we have shown Theorem \ref{thm:invariantMeasures}.

\section{Proofs of Theorems \ref{thm:mainNonExact}, \ref{thm:mainExact}, and \ref{thm:confoliations}} \label{sec:proofs2}

Let $(M, \eta = (\lambda, \omega))$ be a framed Hamiltonian manifold and let
$$\mathbf{u} = (u, C, j, \mathbb{R} \times M, J, D, \mu)$$
be an $\omega$-finite pseudoholomorphic curve. Without loss of generality, we assume that the image of the projection $a \circ u$ of $u$ onto the $\mathbb{R}$-coordinate is unbounded in the positive $\mathbb{R}$-direction.

Suppose that the Hamiltonian vector field $X$ has no periodic orbits, since otherwise the statements of Theorems \ref{thm:mainNonExact}, \ref{thm:mainExact}, and \ref{thm:confoliations} are trivially true.

Then by Theorem \ref{thm:invariantMeasures}, we can assume without loss of generality that there is a sequence $\{t_k\}$ of regular values of $a \circ u$, increasing monotonically without bound such that the sequence of measures $\sigma_k$ defined by
$$\sigma_k(f) = E_\lambda(u, t_k)^{-1}(\int_{(a \circ u)^{-1}(t_k)} u^*(f\lambda))$$
converge weakly to an $X$-invariant probability measure $\sigma$ on $M$. 

For the remainder of this section, fix such a sequence $\{t_k\}$, the measures $\sigma_k$, and the resulting invariant probability measure $\sigma$.

To prove Theorem \ref{thm:mainNonExact} or Theorem \ref{thm:mainExact}, it remains to show that $\sigma$ is not a multiple of the volume density $\lambda \wedge \omega^n$ on $M$.

\subsection{Proof of Theorem \ref{thm:mainNonExact}}
We begin with the proof of Theorem \ref{thm:mainNonExact}. While it does not align with the previous technical setup, the proof is conceptually more clear if one views the volume measure and $\sigma$ as \emph{closed one-dimensional currents} instead. A closed one-dimensional current is a continuous linear functional $T$ on the space of smooth one-forms on $M$ such that, for any exact one-form $\alpha$, 
$$T(\alpha) = 0.$$ 

The framed Hamiltonian structure $\eta = (\lambda, \omega)$ determines a closed, one-dimensional current $V$ by the linear map
$$V: \alpha \mapsto \int_M \alpha \wedge \omega^n$$
where $\alpha$ is any smooth one-form. It is related to the volume measure, in that $V(f\lambda)$ is for any $f \in C^0(M)$ equal to the integral of $f$ on $M$ with respect to the volume measure $\lambda \wedge \omega^n$. 

This current vanishes on exact one-forms by an integration by parts and the fact that $\omega^n$ is closed. As a result, we see that it defines a linear functional $H^1(M; \mathbb{R}) \to \mathbb{R}$ on the first de Rham cohomology group of $M$. By the universal coefficient theorem, the vector space $\text{Hom}(H^1(M; \mathbb{R}), \mathbb{R})$ is isomorphic to $H_1(M; \mathbb{R})$. A careful examination of this chain of identifications shows that the class in $H_1(M; \mathbb{R})$ associated to $V$ is the Poincar\'e dual of the cohomology class of $\omega^n$. 

To define a closed, one-dimensional current version of the measure $\sigma$, we can instead define a sequence of currents $T_k$ by the linear maps
$$T_k: \alpha \mapsto E_\lambda(u, t_k)^{-1} \int_{(a \circ u)^{-1}(t_k)} u^*\alpha.$$

The currents $T_k$ are clearly closed by an application of Stokes' theorem, since they are just re-scalings of integration over the closed level curves $(a \circ u)^{-1}(t_k)$.

An analogue of the compactness result for Radon measures in Theorem \ref{thm:radonMeasureCompactness} shows that, upon passing to a subsequence, there is a closed current $T$ defined by
$$T: \alpha \mapsto \lim_{k \to \infty} T_k(\alpha).$$

We observe three properties of $T$ in the following lemma. 

\begin{lem} \label{lem:propertiesOfInvariantT}
The current $T$ satisfies the following three properties:
\begin{itemize}
\item If $\beta$ is any one-form such that $\beta(X) \equiv 0$, then $T(\beta) = 0$. 
\item $T$ is invariant under the flow of $X$, in the sense that
$$T(\mathcal{L}_X\alpha) = 0$$
for any one-form $\alpha$. 
\item For any one-form $\alpha$, 
$$T(\alpha) = \sigma(\alpha(X)).$$
\end{itemize}
\end{lem}

\begin{proof}
The first item of the lemma is an immediate consequence of the definition of $T$ and Lemma \ref{lem:niceOneForms}. 

By Cartan's formula and the fact that $T$ is closed, the second item is equivalent to showing that
$$T(\iota_X d\alpha) = 0$$
for any smooth one-form $\alpha$. The one-form $\iota_X d\alpha$ contracts with $X$ to zero, so the second item then follows from the first item. 

To prove the third item, we can write any smooth one-form $\alpha$ as $\alpha(X)\lambda + \beta$ where $\beta(X) \equiv 0$. 

Then we compute
\begin{align*}
    T(\alpha) &= T(\alpha(X)\lambda + \beta) \\
    &= T(\alpha(X)\lambda) \\
    &= \lim_{k \to \infty} T_k(\alpha(X)\lambda) \\
    &= \lim_{k \to \infty} E_\lambda(u, t_k)^{-1} \int_{(a \circ u)^{-1}(t_k)} u^*\alpha \\
    &= \lim_{k \to \infty} \sigma_k(\alpha(X)) \\
    &= \sigma(\alpha(X)).
\end{align*}
\end{proof}

Any closed, one-dimensional current has an associated homology class in $H_1(M; \mathbb{R})$, produced by the same procedure we mentioned above in the special case of the current $V$. Any closed, one-dimensional current defines a linear map $H^1(M; \mathbb{R}) \to \mathbb{R}$ on the first de Rham cohomology group of $M$. Then, an application of the universal coefficient theorem shows that $H_1(M; \mathbb{R}) \simeq \text{Hom}(H^1(M; \mathbb{R}), \mathbb{R})$. 

We can show that, unlike $V$, the homology class associated to the limiting current $T$ is zero. The level sets of the curve $u$ all clearly have the same homology class $\alpha$ since any two level sets are co-bounded by a portion of the domain Riemann surface $C$. This implies that the currents $T_k$ have homology class $E_\lambda(u, t_k)^{-1}\alpha$. Since 
$$E_\lambda(u, t_k) \to \infty,$$ it follows that the limiting current $T$ has homology class zero, and therefore cannot be equal to the volume current. 

This discussion, along with Lemma \ref{lem:propertiesOfInvariantT}, proves the second part of Theorem \ref{thm:mainNonExact}. 

We now proceed with the remainder of the proof of Theorem \ref{thm:mainNonExact}, showing that $\sigma$ is not the normalized volume density. 

By Poincar\'e duality, there is a closed one-form $\nu$ such that
$$\int_M \nu \wedge \omega^n \neq 0.$$

Now observe that we can write
$$\nu = q\lambda + \beta$$
where $q = \nu(X)$ and $\beta$ is a one-form such that $\beta(X) \equiv 0$.

The integral of $q$ with respect to the volume measure is clearly
$$\int_M q\lambda \wedge \omega^n = \int_M \nu \wedge \omega^n \neq 0.$$

Recall the definition of the measures $\sigma_k$ from the previous subsection that had $\sigma$ as a weak limit. 

Each of these measures is defined by integrating over the set of loops $(a \circ u)^{-1}(t_k)$ and then re-normalizing by the $\lambda$-energy. 

Let $\pi_M: \mathbb{R} \times M \to M$ be the projection onto the $M$-factor, and 
$$v = \pi_M \circ u: C \to M.$$

Write $\alpha_k \in H_1(M;\mathbb{R})$ for the homology class represented by the image of $(a \circ u)^{-1}(t_k)$ under $v$.

By definition, the cycles represented by the images of $(a \circ u)^{-1}(t_k)$ and $(a \circ u)^{-1}(t_{k+1})$ are cobounded by the image of $(a \circ u)^{-1}([t_k, t_{k+1}])$ under $v$. Therefore, the homology classes $\alpha_k$ and $\alpha_{k+1}$ agree. Since this holds true for any $k$, the homology classes $\alpha_k$ are all equal to a single homology class $\alpha \in H_1(M; \mathbb{R})$. 

Since $\nu$ is closed, the integral over the level loops is homologically determined:
\begin{equation}\label{eq:mainNonExact1}\int_{a \circ u)^{-1}(t_k)} u^*\nu = \langle \nu, \alpha_k \rangle = \langle \nu, \alpha \rangle.\tag{\textbullet}\end{equation}

Since $E_\lambda(u, t_k) \to \infty$, it follows from (\ref{eq:mainNonExact1}) that
$$\lim_{k \to \infty} E_\lambda(u, t_k)^{-1}\int_{(a \circ u)^{-1}(t_k)} u^*\nu = 0.$$

It then follows from Corollary \ref{cor:betaVanishing} that
\begin{align*}
    \sigma(q) &= \lim_{k \to \infty} E_\lambda(u, t_k)\int_{(a \circ u)^{-1}(t_k)} u^*(q\lambda)\\
    &= \lim_{k \to \infty} E_\lambda(u, t_k)\int_{(a \circ u)^{-1}(t_k)} u^*(q \lambda + \beta) \\
    &= \lim_{k \to \infty} E_\lambda(u, t_k)\int_{(a \circ u)^{-1}(t_k)} u^*\nu \\
    &= 0.
\end{align*}

Recall that the integral of $q$ with respect to the volume measure is non-zero.

Since $\sigma$ is non-zero but $\sigma(q) = 0$, it cannot be a constant multiple of the volume measure.

\subsection{Proof of Theorem \ref{thm:mainExact}}
Next, we prove Theorem \ref{thm:mainExact}. In this setting, the non-vanishing of $Lk(\omega)$ provides us a function $q$ such that $\sigma(q) = 0$ but 
$$\int_M q \lambda \wedge \omega^n \neq 0.$$
This implies immediately that $\sigma$ cannot equal the volume density.

We proceed to construct this function, show it has the desired property, and as a result prove Theorem \ref{thm:mainExact}.

Write $\nu$ for a primitive of $\omega$, guaranteed by our exactness assumption. Then $\nu = q\lambda + \beta$, where $\beta$ is a one-form vanishing when contracted with $X$ and $q$ is a smooth function on $M$.

Note that
\begin{align*}
    \int_M q\lambda \wedge \omega^n &= \int_M \nu \wedge \omega^n \\
    &\neq 0.
\end{align*}

The last inequality is given by our assumption that the self-linking number $\text{Lk}(\omega)$ does not vanish. 

On the other hand, we deduce the estimate
\begin{align}
    \sigma(q) &\leq \lim_{k \to \infty} |\sigma_k(q)| \nonumber\\
    &= \lim_{k \to \infty} E_\lambda(t_k)^{-1}|\int_{(a \circ u)^{-1}(t_k)} u^*(q\lambda)| \nonumber\\
    &= \lim_{k \to \infty} E_\lambda(t_k)^{-1}|\int_{(a \circ u)^{-1}(t_k)} u^*(\nu - \beta)| \nonumber\\
    &\leq \lim_{k \to \infty} E_\lambda(t_k)^{-1}|\int_{(a \circ u)^{-1}(t_k)} u^*\nu| \tag{\textbullet} \label{eq:mainExact1}\\
    &\leq \lim_{k \to \infty} E_\lambda(t_k)^{-1}(|\int_{(a \circ u)^{-1}([0, t_k])} u^*\omega| + |\int_{(a \circ u)^{-1}(0)} u^*\nu|) \tag{\textbullet\textbullet} \label{eq:mainExact2} \\
    &\leq \lim_{k \to \infty} \kappa E_\lambda(t_k)^{-1} \tag{\textbullet\textbullet\textbullet} \label{eq:mainExact3} \\
    &= 0 \nonumber.
\end{align}

The inequality (\ref{eq:mainExact1}) is a consequence of the triangle inequality and an application of Corollary \ref{cor:betaVanishing} to remove the $\beta$-term from the integral. The inequality (\ref{eq:mainExact2}) is a consequence of Stokes' theorem and the triangle inequality. 

The constant $\kappa$ in (\ref{eq:mainExact3}) is some constant depending on the curve $u$ and independent of $k$. The fact that $\kappa < \infty$ follows from the fact that the $\omega$-energy of the curve is finite and $\omega$ is non-negative on tangent planes of $C$. 

Taking $k \to \infty$, one finds $\sigma(q) = 0$. Since the integral of $q$ with respect to the volume measure is non-zero, $\sigma$ is not a multiple of the volume measure. 

\subsection{Proof of Theorem \ref{thm:confoliations}}
Finally, we prove Theorem \ref{thm:confoliations}. 

Pick a primitive $\nu$ of $\omega$ such that $\nu(X) \geq 0$, as guaranteed to us by our confoliation type assumption.

Write again $\nu = q\lambda + \beta$ as in the proof of Theorem \ref{thm:mainExact}. Due to the confoliation-type assumption, we have that $q = \nu(X) \geq 0$. 

Then we can show 
$$\text{supp}(\sigma) \subseteq \{q = 0\} \subset M.$$

Recall that the other case of this theorem was handled by Proposition \ref{prop:periodicOrbits}. We are now assuming that there are no periodic orbits, so in particular
$$E_\lambda(u, t_k) \to \infty$$
and Corollary \ref{cor:betaVanishing} holds. 

Write $M_\delta = \{q\geq \delta\}$. Observe that
$$\lambda = q^{-1}(\nu - \beta)$$
wherever $q \neq 0$, so as a consequence of Lemma \ref{lem:nonNegativityLambda} and the fact that $q \geq 0$, $\nu - \beta$ is non-negative on tangent lines of regular level loops wherever $q \neq 0$. 

Now we compute
\begin{align}
    \lim_{k \to \infty} |E_\lambda(t_k)^{-1}\int_{(a \circ u)^{-1}(t_k)\cap M_\delta} u^*\lambda| &= \lim_{k \to \infty} E_\lambda(t_k)^{-1}|\int_{(a \circ u)^{-1}(t_k)\cap M_\delta} q^{-1}u^*(\nu - \beta)| \nonumber\\
    &\leq \lim_{k \to \infty} \delta^{-1} E_\lambda(t_k)^{-1}|\int_{(a \circ u)^{-1}(t_k)\cap M_\delta} u^*(\nu - \beta)| \tag{\textbullet}\label{eq:confoliations1} \\
    &\leq \lim_{k \to \infty} \delta^{-1}E_\lambda(t_k)^{-1}|\int_{(a \circ u)^{-1}(t_k)} u^*(\nu - \beta)| \tag{\textbullet\textbullet}\label{eq:confoliations2}\\
     &\leq \lim_{k \to \infty} \delta^{-1}E_\lambda(t_k)^{-1}|\int_{(a \circ u)^{-1}(t_k)} u^*\nu| \tag{\textbullet\textbullet\textbullet}\label{eq:confoliations3} \\
    &\leq \lim_{k\to\infty}\delta^{-1}E_\lambda(t_k)^{-1}(|\int_{(a \circ u)^{-1}([0, t_k])} u^*\omega| + |\int_{(a \circ u)^{-1}(0)} u^*\nu|) \nonumber\\
    &\leq \delta^{-1}E_\lambda(t_k)^{-1}\kappa \tag{$\star$}\label{eq:confoliations4} \\
    &= 0 \nonumber.
\end{align}

The fact that $\nu - \beta$ is non-negative on the tangent lines of the $(a\circ u)^{-1}(t_k)$ is used to deduce the inequalities (\ref{eq:confoliations1}) and (\ref{eq:confoliations2}). Corollary \ref{cor:betaVanishing} and the triangle inequality are used to deduce (\ref{eq:confoliations3}). The constant $\kappa > 0$ in (\ref{eq:confoliations4}) denotes the sum of the $\omega$-energy of $u$ and the absolute value of the integral of $u^*\nu$ on $(a \circ u)^{-1}(0)$, and is finite due to finiteness of the $\omega$-energy of $u$. 

It follows that the support of $\sigma$ lies outside $M_\delta$ for every $\delta > 0$, so it must lie inside $\{q = 0\}$ as stated in Theorem \ref{thm:confoliations}.

\appendix

\section{Measure-theoretic preliminaries} \label{sec:measureTheory}

We collect in this section various elementary measure-theoretic definitions and results necessary for the proofs of the main theorems.

Fix a closed manifold $M$ of dimension $2n+1$ and a framed Hamiltonian structure $\eta = (\lambda, \omega)$ on $M$. We note that the associated Hamiltonian vector fields $X$ is \emph{volume-preserving} in the following sense.

\begin{lem} \label{lem:volumePreserving}
In any framed Hamiltonian manifold $(M^{2n+1}, \eta = (\lambda, \omega))$, the flow of the Hamiltonian vector field $X$ preserves the volume form
$$\lambda \wedge \omega^n.$$
\end{lem}

The proof of the lemma is an immediate consequence of Cartan's formula for the Lie derivative along with the properties laid out in Definitions \ref{defn:framedHamiltonianStructures} and \ref{defn:hamiltonianVectorField}. 

Now we can observe that from the above lemma that $\lambda \wedge \omega^n$ is an \emph{invariant Radon measure} of $X$. We will now proceed to elaborate on this statement. Recall the following two measure-theoretic definitions.

\begin{defn}
A \textbf{Borel measure} on $M$ is a measure on the algebra generated by Borel sets of $M$. 
\end{defn}

\begin{defn}
A \textbf{Radon measure} on $M$ is a Borel measure $\sigma$ such that
\begin{enumerate}
    \item $\sigma(K) < \infty$ for every compact $K \subseteq M$,
    \item For any subset $A \subseteq M$, $\sigma(A)$ is the infimum of $\sigma(U)$ over all open sets $U$ containing $A$.
    \item For any open $U \subseteq M$, $\sigma(U)$ is the supremum of $\sigma(K)$ over all compact sets $K$ contained in $U$.
\end{enumerate}
\end{defn}

A key property of Radon measures is that we can specify them by continuous linear functionals on $C^0(M)$, the Banach space of continuous functions on $M$. Note that any Radon measure $\sigma$ induces a continuous linear functional on $C^0(M)$ by the action
$$f \mapsto \int_M f d\sigma.$$

The \emph{Riesz representation theorem} below shows that the converse is true as well.

\begin{thm} \emph{(Riesz representation theorem, \cite[Chapter $1$, Theorem $4.14$]{Simon14})} Let $L$ be a continuous functional on $C^0(M)$ such that
$$\sup_{f \in C^0(M),\,|f| \leq 1} L(f) < \infty.$$

Then there is a unique Radon measure $\sigma$ on $M$ such that
$$L(f) = \int_M f d\sigma$$
for every $f \in C^0(M)$.
\end{thm}

The volume form $\lambda \wedge \omega^n$ clearly induces a Radon measure on $M$ by the map
$$f \mapsto \int_M f \lambda \wedge \omega^n.$$

\begin{rem}
Given the Riesz representation theorem above, we often specify without additional comment a Radon measure $\sigma$ by the corresponding linear functional
$$f \mapsto \int_M f d\sigma.$$

This motivates the following notation.
\end{rem}

\begin{defn}
For any Radon measure $\sigma$, we write
$$\sigma(f) = \int_M f d\sigma$$
for any $f \in C^0(M)$.
\end{defn}

The space of Radon measures has a natural norm, given by the dual norm on $C^0(M)^*$:
$$||\sigma|| = \sup_{|f|_{C^0} = 1} |\sigma(f)|.$$

The unit ball in this norm is sequentially compact, which will be quite useful later.

\begin{thm} \label{thm:radonMeasureCompactness}
\cite[Chapter $1$, Theorem $4.16$]{Simon14} Suppose $\sigma_k$ is a sequence of Radon measures on $M$ such that
$$\|\sigma_k\| \leq 1$$
for every $k$. Then after passing to a subsequence, there is a Radon measure $\sigma$ such that
$$\sigma(f) = \lim_{k \to \infty} \sigma_k(f)$$
for every $f \in C^0(M)$. 
\end{thm}

It remains to define an \emph{invariant measure}. This definition is stated for any Borel measure $\sigma$ on $M$.  

Denote the flow of $X$ by the one-parameter family of diffeomorphisms
$$\phi_X^t: M \to M.$$

\begin{defn} \label{defn:invariantMeasure}
Let $\sigma$ be a Borel measure on $M$. We say $\sigma$ is an \textbf{invariant measure} for $X$ if for any $\sigma$-measurable function $f: M \to \mathbb{R}$ and any $t \in \mathbb{R}$, one has
$$\int_M (f \circ \phi_X^t) d\sigma = \int_M f d\sigma.$$
\end{defn}

\begin{exe} \label{exe:invariantMeasureOrbit}
Suppose $X$ has a periodic orbit, parameterized by a smooth map $\gamma: I \to M$ where $I$ is some closed interval. Then the pushforward $\gamma_*\sigma_{Leb}$ of the $1$-dimensional Lebesgue measure is an invariant measure for $X$. 
\end{exe}

\begin{exe} \label{exe:invariantMeasureSet}
More generally, let $S \subseteq M$ be an invariant set containing the \emph{closure} of a trajectory $\gamma: \mathbb{R} \to M$ of $X$. 

Then we can define an invariant measure supported in $S$ by the following procedure. Take some sequence $T_k \to \infty$, and define Radon measures $\sigma_k$ by the functionals
$$f \mapsto 2T_k^{-1}\int_{-T_k}^{T_k} f(\gamma(t))dt.$$ 

Then these measures have a uniform bound on their norm, and we can appeal to Theorem \ref{thm:radonMeasureCompactness} to extract a limit measure $\sigma$, possibly after passing to a subsequence. One can verify easily that $\sigma$ is non-zero, invariant, and supported in $S$. 
\end{exe}

By Definition \ref{defn:invariantMeasure} and Lemma \ref{lem:volumePreserving}, we conclude the following lemma. 

\begin{lem} \label{lem:volumeMeasurePreserving}
The volume form $\lambda \wedge \omega^n$ is invariant under the flow of $X$ as a Radon measure. 
\end{lem}

Although the definition of an invariant measure is given for any Borel measure on $M$, the following lemma gives us an analytically more useful definition for Radon measures. 

\begin{lem} \label{lem:linearizedInvariance}
A Radon measure $\sigma$ is an invariant measure of a vector field $X$ if and only if, for any $C^1$ function $f: M \to \mathbb{R}$, one has
$$\sigma(\nabla_X f) = 0.$$
\end{lem}

\begin{proof}
First, we make the observation that $\sigma$ is invariant if and only if
$$\sigma(f \circ \phi_X^t) = \sigma(f)$$
for any $t > 0$ and any $C^1$ function $f$. 

This follows from the fact that $C^1$ functions are dense in the space of $C^0$ functions on $M$ and the fact that $\sigma$ is a continuous functional. 

Next, suppose that $\sigma$ satisfies the above property. By definition, for any $C^1$ function $f$,
$$\nabla_X f = \lim_{t \to 0} t^{-1}( (f \circ \phi_X^t) - f).$$

Therefore,
$$\sigma(\nabla_X f) = \lim_{t \to 0} t^{-1}(\sigma(f \circ \phi_X^t) - \sigma(f)) = 0.$$

Suppose instead 
$$\sigma(\nabla_X f) = 0$$
for any $C^1$ function $f$.

Note that the continuity and linearity of $\sigma$ implies that if $\{f_s\}_{s \in [0,1]}$ is any one-parameter family of continuous functions on $M$, one has
$$\sigma(\int_0^1 f_s ds) = \int_0^1 \sigma(f_s)ds.$$

We observe that by definition,
$$f \circ \phi_X^t - f = \int_0^1 \nabla_X(f \circ \phi_X^{st}) ds.$$

Then applying the above identity, one finds
\begin{align*}
    \sigma(f \circ \phi_X^t) - \sigma(f) &= \sigma(\int_0^1 \nabla_X(f \circ \phi_X^{st}) ds) \\
    &= \int_0^1 \sigma(\nabla_X(f \circ \phi_X^{st})) ds \\
    &= 0
\end{align*}
as desired. 

\end{proof}

Finally, we can state what it means for a vector field $X$ to be \emph{uniquely ergodic}.

\begin{defn} \label{defn:uniquelyErgodic}
A vector field $X$ is \textbf{uniquely ergodic} if and only if, up to multiplication by a constant, it has a unique non-trivial invariant measure. 
\end{defn}

In the case where $X$ is a Hamiltonian vector field associated to a framed Hamiltonian manifold $(M, \lambda, \omega)$, recall that the volume form $\lambda \wedge \omega^n$ is a non-trivial invariant measure of $X$ (this is the statement of Lemma \ref{lem:volumeMeasurePreserving}). It follows that to show non-unique ergodicity, we need only construct some other non-trivial invariant measure.

\begin{lem} \label{lem:hamiltonianUniquelyErgodic}
The Hamiltonian vector field $X$ associated to a framed Hamiltonian manifold $(M, \lambda, \omega)$ is not uniquely ergodic if and only if there is a non-trivial invariant measure $\sigma$ that is not equal to a constant multiple of $\lambda \wedge \omega^n$. 
\end{lem}

\section{$J$-holomorphic currents and the proof of Proposition \ref{prop:periodicOrbits}} \label{sec:currents}

\subsection{Preliminaries on $J$-holomorphic currents} Since the domain of an $\omega$-finite pseudoholomorphic does not necessarily have finite genus, we must make use of the compactness and regularity theory of \emph{$J$-holomorphic currents} to prove Proposition \ref{prop:periodicOrbits}. 

We collect the relevant definitions and results below. For readers in search of a more expansive exposition, we refer them to \cite{Simon14} for background on rectifiable currents, and \cite{DoanWalpuski21} for a detailed overview of the regularity theory for $J$-holomorphic currents in almost-complex manifolds. Our exposition below follows that of \cite[Section 4]{DoanWalpuski21}. 

First, we write down some basic definitions regarding general currents in smooth manifolds.

\begin{defn}
Let $W$ be a smooth manifold. The space $\Omega^m_c(W)$ denotes the space of smooth, compactly supported $m$-forms on $W$, topologized in the following manner. A sequence $\{\alpha_i\}$ converges in this topology if and only if for sufficiently large $i$, the forms $\alpha_i$ are all supported in some fixed compact set $K$ in $W$, and they converge uniformly in all derivatives on $K$ as $i \to \infty$. 
\end{defn}

If $W$ is compact, then the topology on $\Omega^m_c(W)$ agrees with the standard $C^\infty$ topology on the space $\Omega^m(W)$ of smooth $m$-forms. 

\begin{defn} \label{defn:currents}
A \textbf{$m$-dimensional current} on a smooth manifold $W$ is a real-valued, continuous linear functional on $\Omega^m_c(W)$. The \textbf{boundary} of a $m$-dimensional current $T$ is the $(m-1)$-dimensional current $\partial T$ defined by
$$\partial T(\alpha) = T(d\alpha).$$

A current $T$ is \textbf{closed} if $\partial T \equiv 0$. The \textbf{support} of an $m$-dimensional current is the complement of the maximal open set $U$ with respect to inclusion in $W$ such that 
$$T(\alpha) = 0$$
for any $m$-form $\alpha \in \Omega^m_c(U)$. 
\end{defn}

\begin{defn}
Let $W$ be a smooth manifold. Let $T_k$ be a sequence of $m$-dimensional currents. We say that $T_k$ \textbf{converges weakly} to an $m$-dimensional current $T$ if
$$\lim_{k \to \infty} T_k(\alpha) = T(\alpha)$$
for any $\alpha \in \Omega^k_c(W)$. 
\end{defn}

Next, we write down definitions and results regarding integer rectifiable currents in Riemannian manifolds. If $W$ is a smooth manifold with Riemannian metric $g$, we let $\mathcal{H}^m_g$ denote the $m$-dimensional Hausdorff measure on $W$ associated with the metric $g$. 

\begin{defn}
Let $W$ be a smooth manifold. Choose a Riemannian metric $g$. For any open subset $U$ of $W$, the \textbf{$U$-mass} of a $m$-dimensional current $T$ with respect to $g$, written as $\mathbf{M}_{g,U}(T)$, is defined to be the supremum of $|T(\alpha)|$, taken over all $m$-forms $\alpha \in \Omega^m_c(U)$ with $C^0$ norm bounded above by $1$. 
\end{defn}

\begin{lem} \label{lem:lowerSemicontinuityOfMass}
Let $W$ be a smooth manifold. Let $T_k$ be a sequence of $m$-dimensional currents, converging weakly to an $m$-dimensional current $T$. Then for any proper open subset $U$ of $W$ with compact closure,
$$\mathbf{M}_{g,U}(T) \leq \liminf_{k \to \infty} \mathbf{M}_{g, U}(T_k).$$
\end{lem}

\begin{defn}
Let $W$ be a smooth manifold. Choose a Riemannian metric $g$. A subset $A \subset W$ is \textbf{$m$-rectifiable} if there is a countable set of $m$-dimensional embedded $C^1$ submanifolds $\{M_i\}_{i=1}^\infty$ such that
$$A \subset \cup_{i = 1}^\infty M_i$$
and
$$\mathcal{H}^m_g(\cup_{i=1}^\infty M_i \setminus A) = 0.$$

If $A$ is $m$-rectifiable, then a choice of such a cover by $C^1$ submanifolds produces for $\mathcal{H}^m_g$-almost every $x \in A$ an \textbf{approximate tangent space}, denoted by $T_x A$. The space $T_x A$ is unique if it exists. 
\end{defn}

\begin{defn}
\label{defn:integerRectifiable}
Let $W$ be a smooth manifold equipped with a choice of Riemannian metric $g$. A $m$-dimensional current $T$ is \textbf{integer rectifiable} if for any proper open subset $U$ of $W$ with compact closure, $\mathbf{M}_{g,U}(T) < \infty$ and there exists the following set of data.
\begin{enumerate}
    \item A $m$-rectifiable set $A \subset W$.
    \item An $\mathcal{H}^m_g$-measurable function $\theta$ from $A$ to the nonnegative integers.
    \item An $\mathcal{H}^m_g$-measurable section $\overrightarrow{T}$ of $\Lambda^m TW$ over the set $A$.
\end{enumerate}
We require that, for $\mathcal{H}^m_g$-almost every $x \in A$, the $m$-vector $\overrightarrow{T}$ is given by $\overrightarrow{T}(x) = e_1 \wedge \ldots \wedge e_m$ for any orthonormal frame of $T_x A$, and the current $T$ is defined by the integral
$$T(\alpha) = \int_A \theta(x)\langle \alpha(x), \overrightarrow{T}(x) \rangle d\mathcal{H}^m_g(x).$$
\end{defn}

\begin{thm} \label{thm:federerFlemingCompactness}
\emph{(Federer-Fleming compactness theorem, \cite[Chapter 6, Theorem $3.11$]{Simon14})} Let $W$ be a smooth manifold without boundary, equipped with a Riemannian metric $g$. Let $T_k$ be a sequence of integer rectifiable $m$-dimensional currents, and let $\partial T_k$ denote their boundaries, given as currents on $W$. Suppose that the boundaries $\partial T_k$ are also integer rectifiable. If for any proper open subset $U$ of $W$ with compact closure, the sums
$$\mathbf{M}_{g,U}(T_k) + \mathbf{M}_{g, U}(\partial T_k)$$ 
are uniformly bounded in $k$, then after passing to a subsequence, the sequence $T_k$ converges weakly to an integer rectifiable $m$-dimensional current $T$, with integer rectifiable boundary $\partial T$. 
\end{thm}

\begin{rem} \label{rem:nashEmbedding}
The exact statement of the Federer-Fleming compactness theorem given in \cite[Chapter 6, Theorem $3.11$]{Simon14} concerns the case where $W$ is an open subset of Euclidean space with the Euclidean metric. To prove the statement of Theorem \ref{thm:federerFlemingCompactness} for general manifolds, it suffices to apply the Nash embedding theorem and isometrically embed the Riemannian manifold $(W, g)$ into the Euclidean space $\mathbb{R}^N$. Then we can think of the currents $T_k$ as currents in $\mathbb{R}^N$ and deduce the theorem from the Federer-Fleming compactness theorem in the Euclidean setting. 
\end{rem}

The definitions and results below discuss \emph{semicalibrated} currents and their regularity properties in dimension two. 

\begin{defn}
Let $W$ be a smooth manifold equipped with a choice of Riemannian metric $g$. Let $T$ be a closed, integer rectifiable $m$-dimensional current. The \textbf{regular set} $\text{Reg}(T)$ of $T$ is the largest relatively open set $U$ in the support of $T$ such that the following condition holds. There exists an embedded, oriented $C^1$ $m$-dimensional oriented submanifold $\Sigma \subset W$ and a locally constant function $\theta$ on $\Sigma$ valued in the nonnegative integers such that for any $\alpha \in \Omega^k_c(U)$,
$$T(\alpha) = \int_\Sigma \theta \cdot \alpha.$$

The \textbf{singular set} $\text{Sing}(T)$ of $T$ is the complement in the support of $T$ of $\text{Reg}(T)$. 
\end{defn}

\begin{defn}
Let $W$ be a smooth manifold equipped with a choice of Riemannian metric $g$. A \textbf{$m$-semicalibration} on $W$ is a continuous $m$-form $\sigma$ such that $|\sigma|_{C^0} \leq 1$. A closed, integer rectifiable $m$-dimensional current $T$ is \textbf{semicalibrated} by $\sigma$ if 
$$\langle \sigma(x), \overrightarrow{T}(x) \rangle = 1$$
for $\mathcal{H}^m_g$-almost every point $x$ in the support of $T$. 
\end{defn}

\begin{lem}[Mass bound for semicalibrated currents]
\label{lem:semiCalibratedMass} Let $W$ be a smooth manifold without boundary, equipped with a choice of Riemannian metric $g$. Suppose there exists an $m$-semicalibration $\sigma$. Then a closed, integer rectifiable $m$-dimensional current $T$ that is semicalibrated by $\sigma$ satisfies the following bounds. For any proper open subset $U$ of $W$ with compact closure, and any smooth, nonnegative compactly supported function $\chi$ on $W$ that is equal to $1$ on $U$,
$$\mathbf{M}_{g,U}(T) \leq T(\chi\sigma).$$
\end{lem}

\begin{proof}
Suppose $T$ is given by the data $(A, \theta, \overrightarrow{T})$ as in Definition \ref{defn:integerRectifiable}. 

Suppose that $T$ is semicalibrated by $\sigma$. 

Let $\alpha \in \Omega^k_c(U)$ be an $m$-form compactly supported in $U$ with $C^0$ norm bounded above by $1$. Recall that for $\mathcal{H}^m_g$-almost every $x \in A$, the following two properties are satisfied. First, the $m$-vector $\overrightarrow{T}(x)$ is equal to $e_1 \wedge \ldots \wedge e_m$ for some orthonormal basis $\{e_i\}_{i=1}^m$ of the approximate tangent space $T_x A$. Second, we have the equality
$$\langle \sigma(x), \overrightarrow{T}(x) \rangle = 1.$$

The $C^0$ bound on $\alpha$ then implies that on $\mathcal{H}^m_g$-almost every $x \in A$,
$$|\langle \alpha(x), \overrightarrow{T}(x) \rangle| \leq 1 = \langle \sigma(x), \overrightarrow{T}(x) \rangle.$$

Since $\chi \equiv 1$ on the open set $U$, we deduce the inequality
\begin{align*}
    |T(\alpha)| &\leq \int_A \theta(x) |\langle \alpha(x), \overrightarrow{T}(x) \rangle| d\mathcal{H}^m_g(x) \\
    &= \int_{A \cap U} \theta(x) |\langle \alpha(x), \overrightarrow{T}(x) \rangle| d\mathcal{H}^m_g(x) \\
    &\leq \int_{A \cap U} \theta(x) \langle \sigma(x), \overrightarrow{T}(x) \rangle d\mathcal{H}^m_g(x) \\
    &\leq \int_{A} \theta(x) \langle \chi(x)\sigma(x), \overrightarrow{T}(x) \rangle d\mathcal{H}^m_g(x) \\
    &= T(\chi\sigma).
\end{align*}

Taking the supremum of $|T(\alpha)|$ over all $m$-forms $\alpha \in \Omega^k_c(U)$ with $C^0$ norm bounded above by $1$ yields the desired bound on $\mathbf{M}_{g,U}(T)$. 
\end{proof}

Below we state a powerful regularity theorem for two-dimensional semicalibrated currents, due to De Lellis-Spadaro-Spolaor. It gives an explicit local model for a closed, integer rectifiable two-dimensional semicalibrated current near any point. Roughly, the current is given, for some integer $k \in \mathbb{N}$, by integration with multiplicity over an embedding of the standard $k$-fold branched cover of the two-disk in $\mathbb{R}^{2n+2}$. 

We define this local model below.

\begin{defn} \label{defn:kBranching}
Given $k \in \mathbb{N}$, set
$$\widetilde{D}^k = \{(z,w) \in \mathbb{C}^2\,|\,z = w^k\text{ and }|z| < 1\}.$$

For $\beta \in (0,1)$, let $f: \widetilde{D}^k \to \mathbb{R}^{2n}$ be a continuous injective map which is $C^{3,\beta}$ on $\widetilde{D}^k$, satisfies $f(0) = 0$, and satisfies $|\nabla f(z,w)| \leq \kappa |z|^\alpha$ for some absolute constant $\kappa > 0$. Define an associated map $\underline{f}: \widetilde{D}^k \to \mathbb{R}^{2n+2}$ by
$$\underline{f}(z,w) = (z, f(z,w)).$$

Now let $W$ be a smooth manifold of dimension $2n+2$ and $\phi: U \to W$ a smooth chart where $U$ is an open subset of $\mathbb{R}^{2n+2}$. The \textbf{$k$-branching associated with $f$ and $\phi$} is the integer rectifiable current $G_{f,\phi}$ given by
$$G_{f,\phi}(\alpha) = \int_{\widetilde{D}^k\setminus\{0\}} \underline{f}^*\phi^*\alpha$$
for any $\alpha \in \Omega^2_c(W)$. 
\end{defn}

\begin{thm} \label{thm:semicalibratedRegularity}
\cite{DSS17a, DSS17b} Let $W$ be a smooth manifold of dimension $2n+2$ equipped with a choice of Riemannian metric $g$. Let $T$ be a closed, integer rectifiable $2$-dimensional current, semicalibrated by a $2$-semicalibration $\sigma$. Then the following two properties hold:
\begin{itemize}
    \item For any point $x \in \text{supp}(T) \setminus \text{supp}(\partial T)$, there exists a neighborhood $U$ of $x$ in $W$, a finite collection of charts $\{\phi_i\}_{i=1}^L$ with $\phi_i(0) = x$ for every $i$ and maps $\{f_i\}_{i=1}^L$ as in Definition \ref{defn:kBranching} and weights $\{m_i\}_{i=1}^L$ such that
    $$T|_U = \sum_{i=1}^L m_i G_{f_i,\phi_i}.$$
    \item The singular set $\text{Sing}(T)$ is discrete. 
\end{itemize}
\end{thm}

\begin{rem} \label{rem:onSemicalibratedRegularity} The second statement of Theorem \ref{thm:semicalibratedRegularity} follows from the first. This follows from the fact that a $k$-branching has by definition at most one singular point. 

The second statement is proved in \cite{DSS17a, DSS17b} for $W = \mathbb{R}^{2n+2}$ with the Euclidean metric. It is then proved for any manifold via the Nash embedding theorem as in Remark \ref{rem:nashEmbedding}.

The first statement is not stated explicitly, but is implicit in the work of \cite{DSS17a, DSS17b}. We refer the reader to \cite[Theorem $4.13$]{DoanWalpuski21} for a detailed account of how this statement can be recovered from the work of \cite{DSS17a, DSS17b}. \end{rem}

We conclude by writing down definitions and results regarding two-dimensional $J$-holomorphic currents in almost-Hermitian manifolds. 

\begin{defn} \label{defn:jHolomorphicCurrents}
Let $(W, J,g)$ be an almost-Hermitian manifold. 

A two-dimensional current $T$ is \textbf{$J$-holomorphic} if for any compactly supported two-form $\alpha$ on $W$, we have the identity
$$T(\alpha) = T( \alpha(J-, J-)).$$

Suppose $T$ is integer rectifiable, given by the data $(A, \theta, \overrightarrow{T})$ as in Definition \ref{defn:integerRectifiable}. Then it is $J$-holomorphic if and only if for $\mathcal{H}^2_g$-almost-every point $x$ in the support of $T$, the approximate tangent plane $T_x A$ is $J$-invariant. 
\end{defn}

\begin{lem} \label{lem:jHolomorphicSemicalibration}
Let $(W, J,g)$ be an almost-Hermitian manifold. Write $\sigma$ for the two-form defined by
$$\sigma(-, -) = g(J-, -).$$

Then $\sigma$ is a $2$-semicalibration. A closed, integer rectifiable $2$-dimensional current $T$ is $J$-holomorphic if and only if it is semicalibrated by $\sigma$. 
\end{lem}

\begin{proof}
We observe by the Cauchy-Schwarz inequality that for any two tangent vectors $V$ and $W$ of norm $1$, we have
$$|\sigma(V, W)| = |g(JV, W)| \leq 1$$
with equality if and only if $W = JV$. 

It follows from this inequality that $\sigma$ has $C^0$ norm bounded above by $1$, so it is a $2$-semicalibration. The equality condition implies that a closed, integer rectifiable current $T$ is semicalibrated by $\sigma$ if and only if, for $\mathcal{H}^2_g$-almost every point $x$ in the support of $T$, the approximate tangent plane to $T$ is spanned by the set $\{V, JV\}$ where $V$ is a unit length tangent vector. This is true if and only if the approximate tangent plane is $J$-invariant for $\mathcal{H}^2_g$-almost every point $x$ in the support of $T$, which in turn is true if and only $T$ is $J$-holomorphic. 
\end{proof}

\begin{lem} \label{lem:jHolomorphicCurvesCurrents}
Let $(W, J,g)$ be an almost-Hermitian manifold. Write $\sigma$ for the semicalibration associated to the almost-Hermitian structure $(g, J)$ defined in Lemma \ref{lem:jHolomorphicSemicalibration}. Let
$$\mathbf{v} = (v, C, j, W, J, D, \mu)$$
be a $J$-holomorphic curve in $W$, and let $\theta$ be a locally constant function on $C$ taking values in the positive integers. Then the $2$-dimensional current $T$ defined by
$$T(\alpha) = \int_C \theta \cdot v^*\alpha$$
is integer rectifiable and $J$-holomorphic. 
\end{lem}

\begin{proof}
This is evident from the definition of $T$ and Definitions \ref{defn:integerRectifiable} and \ref{defn:jHolomorphicCurrents}. 
\end{proof}

\begin{lem} \label{lem:jHolomorphicCurrentCompactness}
Let $(W, J,g)$ be an almost-Hermitian manifold. Let $T_k$ be a sequence of closed, integer rectifiable $J$-holomorphic $2$-dimensional currents, converging weakly to a closed, integer rectifiable current $T$. Then $T$ is $J$-holomorphic as well.
\end{lem}

\begin{proof}
This is evident from the definition of weak convergence. For any compactly supported $2$-form $\alpha$, we have 
\begin{align*}
    T(\alpha(J-, J-)) &= \lim_{k \to \infty} T_k(\alpha(J-, J-)) \\
    &= \lim_{k \to \infty} T_k(\alpha) \\
    &= T(\alpha).
\end{align*}

Therefore, $T$ is $J$-holomorphic. 
\end{proof}

\begin{lem} \label{lem:jHolomorphicCurrentRegularity}
Let $(W, J,g)$ be an almost-Hermitian manifold. Let $T$ be a closed, integer rectifiable $J$-holomorphic $2$-dimensional current such that $\mathbf{M}_{g,U}(T) < \infty$ for any proper open subset $U$ of $W$ with compact closure. Then there exists a Riemann surface $(\Sigma, j)$ without boundary, a locally constant function $\theta$ on $\Sigma$ valued in the positive integers, and a $J$-holomorphic curve
$$\mathbf{v} = (v, \Sigma, j, W, J, \emptyset, \emptyset)$$
such that for any compactly supported $2$-form $\alpha$ on $W$,
$$T(\alpha) = \int_C \theta \cdot v^*\alpha.$$
\end{lem}

\begin{proof}
Recall that by Lemma \ref{lem:jHolomorphicSemicalibration}, $T$ is semicalibrated by the two-form
$$\sigma(-, -) = g(J-, -).$$

Then the lemma is a consequence of the regularity theorem for semicalibrated currents in Theorem \ref{thm:semicalibratedRegularity}. 

A proof of the lemma, with the additional assumption that $T$ has compact support, is written down explicitly in \cite[Lemma $5.5$]{DoanWalpuski21}. 

Their proof adapts without difficulty to our setting. The key observation is that the regularity theorems of \cite{DSS17a} and \cite{DSS17b} are local, so they can be applied without assuming $T$ has compact support as long as we assume that $\mathbf{M}_{g, U}(T)$ is finite for any proper open set of $W$ with compact closure. The assumption of compact support is necessary, however, to show that the $J$-holomorphic curve $\mathbf{v}$ in the statement of the theorem has \emph{closed} domain. Such a statement is clearly false in our desired application, the proof of Proposition \ref{prop:periodicOrbits}, and not required.  

First, let $\dot\Sigma = \text{Reg}(T)$ be the set of regular points of $T$, which by definition is an embedded $C^1$ submanifold of $W$. Since $T$ is semicalibrated by $\sigma$, it follows that $\dot\Sigma$ has $J$-invariant tangent planes. 

Elliptic regularity shows that $\dot\Sigma$ is a smooth submanifold of $W$. This is somewhat subtle, since a priori $\dot\Sigma$ is the image of a pseudoholomorphic embedding of a Riemann surface with a \emph{continuous} almost-complex structure. The ``measurable Riemann mapping theorem'' of Ahlfors-Bers \cite{AhlforsBers60} can be used on the domain to show that, at least locally, the pseudoholomorphic map takes the form of an embedding
$$u: (D, i) \to (W, J)$$
satisfying the equation
$$\partial_s u + J(u)\partial_t u$$
where $D$ is the unit disk in the complex plane with real coordinates $(s,t)$ and $J(u)$ is continuous. From here it is clear from standard theory that the map is in fact smooth. 

The regularity theorem written down in Theorem \ref{thm:semicalibratedRegularity} states that $\text{Sing}(T)$ is discrete.

Moreover, the finer description of $T$ given by the first item in  Theorem \ref{thm:semicalibratedRegularity} shows each $x \in \text{Sing}(T)$ has a neighborhood $V_x$ with the following properties. First, the intersection of $\dot\Sigma \cap V_x$ is diffeomorphic to a finite disjoint union 
$$\dot\Sigma \cap V_x \simeq D \setminus \{0\} \sqcup \ldots \sqcup D\setminus\{0\}$$
where $D$ denotes the open unit disk in the complex plane. Second, 
$$\text{Sing}(T) \cap V_x = \{x\}.$$

Since $W$ is locally compact, we may also assume that $V_x$ has compact closure in $W$. 

It follows that we can compactify $\dot\Sigma$ in each such neighborhood $V_x$ by adding finitely many points. This yields a Riemann surface $(\Sigma, j)$, a holomorphic embedding of Riemann surfaces $\dot\Sigma \hookrightarrow \Sigma$ and a continuous map $v: \Sigma \to W$ extending the smooth $J$-holomorphic embedding $\dot\Sigma \hookrightarrow W$. 

For any $z \in \Sigma \setminus \dot\Sigma$, recall that the neighborhood $V_{v(z)}$ has compact closure. By our assumptions we have $\mathbf{M}_{g,V_x}(T) < \infty$. Now recall that, by definition, $v^{-1}(V_{v(z)} \setminus \{v(z)\})$ is a finite union of punctured disks of the form $D\setminus\{0\}$. Equip these disks with the Euclidean metric. Then the mass bound implies immediately that the map $v$ has finite energy on $v^{-1}(V_{v(z)} \setminus \{v(z)\})$:
$$\int_{v^{-1}(V_{v(z)}\setminus \{v(z)\})} |Dv|^2 < \infty.$$

By the removable singularity theorem \cite[Theorem $4.2.1$]{McduffSalamon12}, $v$ extends to a smooth, $J$-holomorphic map on each connected component of $v^{-1}(V_{v(z)} \setminus \{v(z)\})$. It follows that the map $v: \Sigma \to W$ is smooth and $J$-holomorphic, and so the tuple
$$\mathbf{v} = (v, \Sigma, j, W, J, \emptyset, \emptyset)$$
forms the datum of a $J$-holomorphic curve. 

Recall that there is a locally constant function $\theta$ on $\dot\Sigma$ valued in the positive integers such that for any two-form $\alpha$ with pullback $v^*\alpha$ compactly supported in $\dot\Sigma$, 
$$T(\alpha) = \int_{\dot\Sigma} \theta \cdot v^*\alpha.$$

The function $\theta$ extends uniquely to a locally constant function, also denoted by $\theta$, on $\Sigma$ valued in the positive integers. Since the map $v: \Sigma \to W$ is smooth and $J$-holomorphic, it follows that
$$T(\alpha) = \int_{\Sigma} \theta \cdot v^*\alpha$$
for \emph{any} compactly supported two-form in $W$. 
\end{proof}

\begin{rem} Let $(W, J,g)$ be an almost-Hermitian manifold. Suppose $(W, J, g)$ is \emph{locally symplectic}, i.e. for any point $x \in W$, there is a neighborhood $U$ of $x$ and a symplectic form $\Omega$ on $U$ such that $J$ is compatible with $\Omega$. Then the results of \cite{DSS17a} and \cite{DSS17b} for $2$-dimensional semicalibrated currents are not necessary, and we can appeal instead to the regularity theorems for $2$-dimensional $J$-holomorphic currents proved in \cite{RiviereTian09}. The authors of \cite{RiviereTian09} note furthermore that if $W$ is four-dimensional, then $(W,g,J)$ is locally symplectic. 
\end{rem}

Before we prove Proposition \ref{prop:periodicOrbits}, we require two more geometric propositions regarding pseudo-holomorphic curves in almost-Hermitian manifolds. 

Recall that we have fixed at the beginning of this section a framed Hamiltonian manifold $(M, \eta = (\lambda, \omega))$ and an $\eta$-adapted cylinder $(\mathbb{R} \times M, J, g)$. 

The first is an area bound for pseudoholomorphic curves in the adapted cylinder depending on the ``height'' of the curve and the $\lambda$-energy of the curve at a single level $t \in \mathbb{R}$.

This theorem was proved in \cite{FishHoferFeral} in greater generality. We discuss this theorem in detail in \cite{sequel}. 

\begin{thm}
\label{thm:exponentialAreaBoundsBasic} \cite[Theorem $3$]{FishHoferFeral} Fix some $a_0 \in \mathbb{R}$,  positive constants $\epsilon$ and $\Lambda$, and a $\omega$-finite curve
$$\mathbf{u} = (u, C, j, \mathbb{R} \times M, J, D, \mu)$$
such that
\begin{itemize}
    \item $a_0$ is a regular level of $a \circ u$,
    \item $E_\lambda(u, a_0) \leq \Lambda$,
    \item $(a \circ u)(C) = [a_0 - \epsilon, a_0 + \epsilon]$ and $(a \circ u)(\partial C) = \{a_0 \pm \epsilon\}$. 
\end{itemize}

Then there is a constant $\kappa = \kappa(\epsilon, \Lambda, M, \eta, g, J) > 0$ such that
$$\text{Area}_{u^*g}(C) \leq \kappa.$$
\end{thm}

The second required result is essentially the standard monotonicity formula for pseudoholomorphic curves in symplectic manifolds, but for the more general case of pseudoholomorphic curves in almost-Hermitian manifolds. 

\begin{prop} \label{prop:monotonicity}
\cite[Corollary $3.7$]{FishCurves} 

For any $\epsilon > 0$, there exists $r_0 = r_0(\epsilon, M, g, J) > 0$ such that the following holds for any $r \in (0, r_0)$. Let
$$\mathbf{u} = (u, C, j, \mathbb{R} \times M, J, D, \mu)$$
be a pseudoholomorphic curve such that $C$ is compact and the restriction of $u$ to any component of $C$ is not constant. Then for any $z \in C$ such that $u^{-1}(B_r(u(z)))$ does not intersect the boundary $\partial C$ of the domain, we find
$$\text{Area}_{u^*g}(u^{-1}(B_r(u(z)))) \geq (1 + \epsilon)^{-1}\pi r^2.$$
\end{prop}

\subsection{The proof of Proposition \ref{prop:periodicOrbits}} We are now equipped with all of the necessary tools to prove  Proposition \ref{prop:periodicOrbits}.

Pick any unbounded sequence of regular values $\{t_k\}$ of $a \circ u$ such that $u(C)$ intersects $\{t_k\} \times M$ for every $k$, and $E_\lambda(u, t_k)$ is uniformly bounded. Without loss of generality, after passing to a subsequence we may assume that $t_k$ either increases monotonically to $\infty$ or decreases monotonically to $-\infty$. We will prove Proposition \ref{prop:periodicOrbits} assuming the former holds, as the proof in the latter case is identical. 

It follows that we can extract a subsequence such that $E_\lambda(u, t_k)$ converges. 

Fix any $\epsilon > 0$ such that for any $k$, both $t_k + \epsilon$ and $t_k - \epsilon$ are regular values of $a \circ u$.

Then it follows from the exponential area estimate in Theorem \ref{thm:exponentialAreaBoundsBasic} that the restrictions of $u$ to the compact sub-surfaces
$$C_k = (a \circ u)^{-1}([t_k - \epsilon, t_k + \epsilon])$$
have uniformly bounded area in $k$. 

The fact that the map $u$ is proper, $u(C)$ has finitely many connected components, and the domain $C$ has at least one non-compact connected component as well as finitely many boundary components implies the following statements are true for any sufficiently large $k$. 

First, the image under $u$ of the boundaries $\partial C_k$ lie in $\{t_k \pm \epsilon\} \times M$. Second, $C_k$ has at least one connected component $C_k^0 \subset C_k$ such that $u(C_k^0)$ crosses the entire cylinder $[t_k - \epsilon, t_k + \epsilon] \times M$, that is $(a \circ u)(C_k^0) = [t_k - \epsilon, t_k + \epsilon]$. 

We assume $k$ is sufficiently large so that these statements hold for any $k$. 

Define the $J$-holomorphic map $u_k$ to be the composition of the restriction of $u$ to $C_k$ and a shift by $-t_k$ in the $\mathbb{R}$-coordinate.

This yields a sequence 
$$\mathbf{u}_k = (u_k, C_k, j_k, [-\epsilon, \epsilon] \times M, J, \emptyset, \emptyset)$$
of $J$-holomorphic curves mapping into $[-\epsilon, \epsilon] \times M$ with uniformly bounded area. 

Write $T_k$ for the two-dimensional, integral rectifiable $J$-holomorphic current on the non-compact cylinder $(-\epsilon, \epsilon) \times M$ given by
$$T_k(\alpha) = \int_{C_k} u_k^*\alpha$$
for any smooth, compactly supported two-form $\alpha$ on $(-\epsilon, \epsilon) \times M$. 

\begin{rem}
We restrict our currents to the non-compact manifold $(-\epsilon, \epsilon) \times M$ for a couple of different reasons. 

The first reason is that we wish to apply the Federer-Fleming compactness theorem to the $T_k$ to extract a limiting $J$-holomorphic current $T$ from a subsequence. The Federer-Fleming compactness theorem, as stated in Theorem \ref{thm:federerFlemingCompactness} requires the ambient manifold to not have boundary. Furthermore, observe that the boundaries of the currents $T_k$ are supported in $\{\pm \epsilon\} \times M$. Therefore, their restrictions to $(-\epsilon, \epsilon) \times M$ have no boundary, as rigorously proved in Lemma \ref{lem:currentsAreAllClosed} below. This means that is not necessary to control the local masses of the boundaries $\partial T_k$ to apply the Federer-Fleming compactness theorem, only the local masses of the currents $T_k$ themselves.

The second reason is that, after application of the Federer-Fleming compactness theorem, we want to apply the regularity theory of \cite{DSS17a, DSS17b} to show that the limiting current $T$ is in fact given by integration over a $J$-holomorphic curve. The regularity theorem is an \emph{interior regularity} theorem, meaning it only holds away from the support of the boundary $\partial T$. Therefore, it is simplest to restrict the currents $T_k$ to $(-\epsilon, \epsilon) \times M$ so that they have no boundary and the limit current $T$ has no boundary, so the regularity theorem holds everywhere. 

We do note that a $J$-holomorphic curve representing $T$ could have singularities accumulating at the boundary of $[-\epsilon, \epsilon] \times M$. This is not an issue in our arguments. 
\end{rem}

The following two lemmas show that the sequence $T_k$ satisfies the assumptions of the Federer-Fleming compactness theorem stated in Theorem \ref{thm:federerFlemingCompactness}. 

\begin{lem} \label{lem:currentsAreAllClosed}
The currents $T_k$ are closed.
\end{lem}

\begin{proof}
By Stokes' theorem, for any $k$ and any compactly supported $1$-form $\alpha$ on $(-\epsilon, \epsilon) \times M$ we have
\begin{align*}
    \partial T_k(\alpha) &= T_k(d\alpha) \\
    &= \int_{C_k} u_k^*d\alpha \\
    &= \int_{C_k} d(u_k^*\alpha) \\
    &= \int_{\partial C_k} u_k^*\alpha \\
    &= 0.
\end{align*}
The last line follows from the fact that $\alpha$ is compactly supported and $u_k(\partial C_k)$ lies in $\{\pm \epsilon\} \times M$. 
\end{proof}

\begin{lem} \label{lem:uniformMassBound}
For any proper open subset $U$ of $(-\epsilon, \epsilon) \times M$ with compact closure, the masses $\mathbf{M}_{g, U}(T_k)$ are uniformly bounded.
\end{lem}

\begin{proof}
To prove the lemma, we make two observations. First, since the currents $T_k$ are $J$-holomorphic, by Lemma \ref{lem:jHolomorphicSemicalibration} they are semicalibrated by the two-form
$$g(J-, -) = da \wedge \lambda + \omega.$$

Second, by Lemma \ref{lem:semiCalibratedMass}, for any proper open subset $U$ of $(-\epsilon, \epsilon) \times M$ with compact closure and smooth, nonnegative function $\chi$ on $(-\epsilon, \epsilon) \times M$ that is compactly supported and equal to $1$ on $U$, we have the inequality
$$\mathbf{M}_{g,U}(T_k) \leq T_k(\chi \cdot (da \wedge \lambda + \omega)).$$

However, by definition
$$T_k(\chi \cdot (da \wedge \lambda + \omega)) \leq \text{Area}_{u_k^*g}(C_k).$$

The uniform mass bound follows immediately by combining these two inequalities, along with the fact that $\text{Area}_{u_k^*g}(C_k)$ is uniformly bounded above in $k$. 
\end{proof}

We apply the Federer-Fleming compactness theorem (written as Theorem \ref{thm:federerFlemingCompactness} above) and Lemma \ref{lem:jHolomorphicCurrentCompactness}. The currents $T_k$ converge weakly after passing to a subsequence to a closed, two-dimensional, $J$-holomorphic integer rectifiable current $T$ on $(-\epsilon, \epsilon) \times M$.

We show in the following three lemmas that the support of $T$ is quite restricted.

\begin{lem} \label{lem:nonzeroSupport}
The current $T$ is not zero. 
\end{lem}

\begin{proof}
Recall that each $C_k$ has at least one connected component $C_k^0 \subset C_k$ such that $(a \circ u_k)(C_k^0) = [-\epsilon, \epsilon]$. 

It follows that for every $k$ there is a point $z_k \in u_k(C_k^0)$ such that 
$$a(z_k) = 0$$
and the ball $B_{\epsilon/100}(z_k)$ does not intersect $u_k(\partial C_k^0)$. 
Therefore, the monotonicity formula as in Proposition \ref{prop:monotonicity} implies that there is some $k$-independent constant $\delta \in (0, \epsilon/100)$ depending only on $\epsilon$, $M$, and the choice of almost-Hermitian structure $(J, g)$ on $\mathbb{R} \times M$ such that $$\text{Area}_{u_k^*g}(u_k^{-1}(B_\delta(z_k)) \geq \frac{\pi}{2}\delta^2$$ for every $k$. After passing to a subsequence, we may assume that the $z_k$ converge to some point $z \in \{0\} \times M$, and $z_k \in B_{\delta}(z)$ for every $k$. 

Now pick a smooth cutoff function 
$$\chi: (-\epsilon, \epsilon) \times M \to [0,1]$$
such that $\chi \equiv 1$ on $B_{\epsilon/25}(z)$ and $\chi \equiv 0$ outside of $B_{\epsilon/10}(z)$. 

It follows by definition that, for every $k$,
\begin{align}
    T_k(\chi \cdot (da \wedge \lambda + \omega)) &= \int_{C_k} u_k^*(\chi \cdot (da \wedge \lambda + \omega)) \nonumber \\
    &\geq \int_{u_k^{-1}(B_{\epsilon/25}(z))} u_k^*(da \wedge \lambda + \omega) \label{eq:nonzeroSupport1} \tag{\textbullet}\\
    &\geq \int_{u_k^{-1}(B_{\delta}(z_k))} u_k^*(da \wedge \lambda + \omega) \label{eq:nonzeroSupport2} \tag{\textbullet\textbullet} \\
    &= \label{eq:nonzeroSupport3} \tag{\textbullet\textbullet\textbullet} \text{Area}_{u_k^*g}(u_k^{-1}(B_{\delta}(z_k))) \\
    &\geq \frac{\pi}{2}\delta^2. \nonumber
\end{align}

The inequality (\ref{eq:nonzeroSupport1}) makes use of the fact that $\chi \equiv 1$ on $B_{\epsilon/25}(z)$, and the inequality (\ref{eq:nonzeroSupport2}) makes use of the fact that $z_k$ is a distance of at most $\delta < \epsilon/100$ from $z$. The equality (\ref{eq:nonzeroSupport3}) uses the fact that $u_k^*(da \wedge \lambda + \omega)$ agrees with the volume form of the pullback metric $u_k^*g$. 

It now follows that
$$T(\chi \cdot (da \wedge \lambda + \omega)) = \lim_{k \to \infty} T_k(\chi \cdot (da \wedge \lambda + \omega)) \geq \frac{\pi}{2}\delta^2 > 0$$
so $T$ is not equal to zero. 

\end{proof}

Next, we show that $T$ has support invariant under translation in the $\mathbb{R}$-coordinate. Precisely, if $(a, p) \in (-\epsilon, \epsilon) \times M$ lies in the support of $T$, we will show that $(b, p)$ lies in the support of $T$ for any other $b \in (-\epsilon, \epsilon)$. This is equivalent to showing that, for any compactly supported two-form $\alpha \in \Omega^2_c( (-\epsilon, \epsilon) \times M)$, if we define for all sufficiently small $s \in \mathbb{R}$ the translated two-form
$$\tau_s^*\alpha(a, p) = \alpha(a + s, p)$$
that 
$$T(\tau_s^*\alpha) = T(\alpha).$$

Linearizing this condition shows that it is equivalent to the identity $$T(\mathcal{L}_{\partial_a}\alpha) = 0$$ where $\mathcal{L}$ denotes the Lie derivative. 

\begin{lem} \label{lem:translationInvariantSupport}
The support of $T$ is invariant under translation in the $\mathbb{R}$-coordinate. 
\end{lem}

\begin{proof}
From the discussion above the statement of the lemma, recall that it suffices to show that, for any compactly supported two-form 
$$\alpha \in \Omega^2_c((-\epsilon, \epsilon) \times M)$$
that 
$$T(\mathcal{L}_{\partial_a}\alpha) = 0.$$

Since $T$ is closed, this is equivalent by Cartan's formula to showing
$$T(\iota_{\partial_a} d\alpha) = 0.$$

Now write $\alpha = \alpha_0 + da \wedge \alpha_1$ where $\alpha_0$ is a two-form that contracts with $\partial_a$ to zero, and $\alpha_1$ is a one-form that contracts with $\partial_a$ to zero.

We then compute for any pair of tangent vectors $V$ and $W$ that
$$(\iota_{\partial_a} d\alpha)(V, W) = d\alpha_0(\partial_a, V, W) - d^M\alpha_1(V, W).$$

Here $d^M\alpha_1$ denotes the exterior derivative along $M$, so in particular $d^M\alpha_1$ contracts with $\partial_a$ to zero. We use this computation to estimate $T_k(\iota_{\partial_a}d\alpha)$ for every $k$. 

By definition,
\begin{align*}
    T_k(\iota_{\partial_a}d\alpha) &= \int_{C_k} u_k^*(\iota_{\partial_a}d\alpha) \\
    &= \int_{C_k} (\iota_{\partial_a}d\alpha)(du_k(e), Jdu_k(e)) \text{dvol}_{u_k^*g}
\end{align*}
where $e$ denotes any tangent vector to $C_k$ such that $|du_k(e)|_g = 1$. 

Now let $V$ be any tangent vector in the target $[-\epsilon, \epsilon] \times M$. We can decompose $V$ orthogonally as 
$$V = da(V)\partial_a + \lambda(V)X + \pi_\xi V$$
where $\pi_\xi$ denotes orthogonal projection onto the tangent distribution $\xi$.

By our earlier computation, for any tangent vector $V$ of length $1$ in the target $[-\epsilon, \epsilon] \times M$, we find
\begin{align*}
    (\iota_{\partial_a} d\alpha)(V, JV) &= d\alpha_0(\partial_a, V, JV) - d^M\alpha_1(V, JV) \\
    &= d\alpha_0(\partial_a, \lambda(V)X + \pi_\xi V, da(V)X + J\pi_\xi V) \\
    &\qquad + d^M\alpha_1(\lambda(V)X + \pi_\xi V, da(V)X + J\pi_\xi V) \\
    &\leq (\|d\alpha_0\|_{C^0} + \|d^M\alpha_1\|_{C^0})[(\lambda(V) + da(V) + \|\pi_\xi V\|_g)\|\pi_\xi V\|_g \ \\
    &\leq 10\|\alpha\|_{C^1}\|\pi_\xi V\|_g \\
    &= 10\|\alpha\|_{C^1}\omega(V, JV)^{1/2}.
\end{align*}

The second inequality follows from the fact that $\|V\|_g = 1$, which implies $\lambda(V) + da(V) + \|\pi_\xi V\|_g \leq 10$. 

It follows by applying the above derivation with $V = du_k(e)$ and applying the H\"older inequality that
\begin{align*}
    T_k(\iota_{\partial_a}d\alpha) &\leq 10|\alpha|_{C^1} \int_{C_k} \omega(du_k(e), Ju_k(e))^{1/2} \text{dvol}_{u_k^*g} \\
    &\leq 10\|\alpha\|_{C^1}(\int_{C_k} u^*\omega)^{1/2}\text{Area}_{u_k^*g}(C_k) \\
    &\leq 10\|\alpha\|_{C^1}\text{Area}_{u_k^*g}(C_k)E_\omega(u_k).
\end{align*}

Since the original map $u$ has finite $\omega$-energy, it follows that $E_\omega(u_k) \to 0$. We also know that $\text{Area}_{u_k^*g}(C_k)$ is uniformly bounded, so it follows that
$$\lim_k T_k(\iota_{\partial_a}d\alpha) = 0.$$

This implies that $T(\iota_{\partial_a}d\alpha) = 0$ which implies the lemma. 
\end{proof}

We conclude by showing that $T$ has support invariant under the flow of the Hamiltonian vector field $X$. 

\begin{lem} \label{lem:flowInvariantSupport}
The support of $T$ is invariant under the flow of $X$.
\end{lem}

\begin{proof}
The proof is similar to that of Lemma \ref{lem:translationInvariantSupport}, but now we need to show that 
$$T(\iota_X d\alpha) = 0$$
for any compactly supported two-form $\alpha$. 

We split up $\alpha$ as in the proof of Lemma \ref{lem:translationInvariantSupport} and compute 
$$\iota_X d\alpha(V, W) = \iota_X d\alpha_0(V, W) + (da \wedge \iota_X d^M\alpha_1)(V, W).$$

This implies for any unit length tangent vector $V$ that
$$\iota_X d\alpha(V, JV) \leq 10\|\alpha\|_{C^1}\|\pi_\xi V\|_g$$
as before. 

This bound then shows that $T(\iota_X d\alpha) = 0$ as desired. 
\end{proof}

Next, we apply the regularity theory of \cite{DSS17a, DSS17b} to show that $T$ is given by integration with multiplicity over a $J$-holomorphic curve.

By Lemma \ref{lem:jHolomorphicCurrentRegularity}, there is a Riemann surface $(\Sigma, j)$, a smooth $J$-holomorphic map $v$ from $\Sigma$ to $(-\epsilon, \epsilon) \times M$ and a locally constant, positive integer-valued function $\theta$ on $\Sigma$ such that for any two-form $\alpha$ compactly supported in $(-\epsilon, \epsilon) \times M$, 
$$T(\alpha) = \int_{\Sigma} \theta \cdot v^*\alpha.$$

Observe that the support of $T$ is contained in the image of the pseudoholomorphic map $v$. By Lemmas \ref{lem:nonzeroSupport}, \ref{lem:translationInvariantSupport} and \ref{lem:flowInvariantSupport}, it follows that the image $v(\Sigma)$ contains at least one subset of the form $(-\epsilon, \epsilon) \times \gamma$, where $\gamma$ is a trajectory of the Hamiltonian vector field $X$.

However, if $\gamma$ is not periodic, the intersection of this subset with any open set 
$$U_\delta = (-(\epsilon - \delta), \epsilon - \delta) \times M,$$
with $\delta < \epsilon$, equal to 
$$(-(\epsilon - \delta), (\epsilon - \delta)) \times \gamma$$
has infinite two-dimensional Hausdorff measure. On the other hand, the intersection of the image $v(\Sigma)$ of the $J$-holomorphic map $v$ with $U_\delta$ has finite two-dimensional area. By definition,
$$\text{Area}_{v^*g}(v^{-1}(v(\Sigma) \cap U_\delta)) \leq \mathbf{M}_{g, U_\delta}(T).$$

Then it suffices to show that $\mathbf{M}_{g,U_\delta}(T)$ is finite. This follows from Lemma \ref{lem:uniformMassBound}, which shows that $\mathbf{M}_{g, U_\delta}(T_k)$ is uniformly bounded, and the definition of weak convergence of currents. Therefore, $\gamma$ must be periodic and we have shown that $X$ has a periodic orbit, and proved Proposition \ref{prop:periodicOrbits}.

\section{A weaker version of Corollary \ref{cor:uniqueErgodicity}} \label{sec:weakerNonExact}

In this section, we prove the following proposition regarding non-unique ergodicity of the Hamiltonian vector field $X$ of a non-exact framed Hamiltonian manifold. This can be regarded as a weaker version of Corollary \ref{cor:uniqueErgodicity}. 

\begin{prop} \label{prop:mainNonExactDim3}
Suppose $(M^{2n+1}, \eta = (\lambda, \omega))$ is a framed Hamiltonian manifold such that
\begin{enumerate}
    \item $M$ is of dimension three and not a $T^2$ bundle over $S^1$ \emph{or} $M$ is of any odd dimension greater than three and does not fiber over the circle,
    \item $\omega^n$ is not exact.
\end{enumerate}

Then the Hamiltonian vector field $X$ is not uniquely ergodic. 
\end{prop}

Before we begin the proof of Proposition \ref{prop:mainNonExactDim3}, we will fix some notation and definitions. 

\begin{defn}[Mapping torus]
Let $\Sigma$ be a closed, even-dimensional manifold equipped with a volume form $\theta$. Suppose there is a diffeomorphism
$$\Phi: \Sigma \to \Sigma$$
such that $\Phi$ preserves the volume, i.e. $\Phi^*\theta = \theta$. 

Then the \emph{mapping torus} of $\Phi$ is the three-manifold
$$M_\Phi = [0,1] \times \Sigma / \sim_\Phi$$
where the equivalence relation $\sim_\Phi$ is defined by
    $$(p, 1) \sim (\Phi(p), 0)$$
for every $p \in \Sigma$.
\end{defn}

The mapping torus $M_\Phi$ of a diffeomorphism $\Phi$ of an even-dimensional smooth manifold $\Sigma$ preserving a volume form $\theta$ has a natural, non-vanishing vector field defined as follows. Write $t$ for the $[0,1]$-coordinate on $[0,1] \times \Sigma$. Then, if
$$\pi: [0,1] \times \Sigma \to M_\Phi$$
is the natural projection, the pushforward vector field $\pi_*\partial_t$ is a non-vanishing vector field on $M_\Phi$. 

For brevity, we will also denote this vector field by $\partial_t$. Note that by definition, there is a one-to-one correspondence between periodic orbits of $\partial_t$ on $M_\Phi$ and periodic orbits of the diffeomorphism $\Phi$. 

The first step of the proof of Proposition \ref{prop:mainNonExactDim3} is the following proposition.

\begin{prop} \label{prop:actuallyAMappingTorus}
Let $(M^{2n+1}, \lambda, \omega)$ be a closed framed Hamiltonian manifold such that $\omega^n$ is not exact and the Hamiltonian vector field $X$ is uniquely ergodic. There is a closed manifold $\Sigma$ of dimension $2n$ with volume form $\theta$ and a volume-preserving diffeomorphism 
$$\Phi: \Sigma \to \Sigma$$
such that $M$ is diffeomorphic to the mapping torus $M_\Phi$. 

Moreover, a diffeomorphism 
$$h: M_\Phi \to M$$
can be chosen such that the pushforward $h_*\partial_t$ is equal to the Hamiltonian vector field $X$.
\end{prop}

We will defer the proof of Proposition \ref{prop:actuallyAMappingTorus} to Section \ref{subsec:mappingTorusProp}. The proof is an application of a classical result of Schwartzman \cite{Schwartzman57, Sullivan76}, written out below in Proposition \ref{prop:asymptoticCycles}.

We will now show the following lemma. This lemma enables us to address any $3$-manifold that is not a $T^2$-bundle over $S^1$ in Proposition \ref{prop:mainNonExactDim3}, rather than merely $3$-manifolds that do not fiber over $S^1$. 

\begin{lem} \label{lem:weinsteinStableHamiltonian}
Let $M_\Phi$ be a mapping torus of a surface diffeomorphism
$$\Phi: \Sigma \to \Sigma$$
that preserves a volume form $\theta$ on $\Sigma$. Then, if $M_\Phi$ is not a $T^2$-bundle over $S^1$, the vector field $\partial_t$ has a periodic orbit.
\end{lem}

\begin{proof}
We can give $M_\Phi$ a natural framed Hamiltonian structure for which $\partial_t$ is the Hamiltonian vector field. Recall the volume form $\theta$ on the surface $\Sigma$ preserved by the diffeomorphism $\Phi$. 

Then $\theta$ pulls back to $[0,1] \times \Sigma$, and since it is preserved by $\Phi$, it descends to a closed, nowhere vanishing two-form on $M_\Phi$, which we denote by $\omega$. 

We can also define a closed one-form $dt$ on $M_\Phi$ by pushing forward the corresponding one-form $dt$ on $[0,1] \times \Sigma$ to $M_\Phi$. 

Then it is immediate that $(dt, \omega)$ is a framed Hamiltonian structure with Hamiltonian vector field $\partial_t$. 

Furthermore, $(dt, \omega)$ is a \emph{stable} Hamiltonian structure. Recall that a stable Hamiltonian structure is a framed Hamiltonian structure $(\lambda, \omega)$ such that 
$$\ker(\omega) \subseteq \ker(d\lambda).$$

In this case, $dt$ is closed, so it is immediate that $(dt, \omega)$ is stable. 

Hutchings-Taubes have shown in \cite{HutchingsTaubes09} that, as long as the stable Hamiltonian manifold $(M_\Phi, dt, \omega)$ is not a $T^2$-bundle over $S^1$, then the Hamiltonian vector field $\partial_t$ has a periodic orbit.

We have a priori assumed that $M$ is not a $T^2$-bundle over $S^1$, so $M_\Phi$ is not as well, and the proposition is proved. 
\end{proof}

Proposition \ref{prop:actuallyAMappingTorus} and Lemma \ref{lem:weinsteinStableHamiltonian} immediately prove Proposition \ref{prop:mainNonExactDim3}. 

\begin{proof}[Proof of Proposition \ref{prop:mainNonExactDim3}]

Let $(M, \eta = (\lambda, \omega))$ be a closed, framed Hamiltonian manifold satisfying the conditions of Proposition \ref{prop:mainNonExactDim3}.

Assume for the sake of contradiction that the Hamiltonian vector field, denoted by $X$, is uniquely ergodic.

Then Proposition \ref{prop:actuallyAMappingTorus} shows that $M$ fibers over $S^1$ because it is diffeomorphic to a mapping torus. If $M$ has dimension greater than or equal to five, then by assumption $M$ does not fiber over $S^1$ so we arrive at a contradiction and $X$ cannot be uniquely ergodic in this case.

If $M$ has dimension three, then observe that Lemma \ref{lem:weinsteinStableHamiltonian} shows that the Hamiltonian vector field $X$ has a periodic orbit if and only if the vector field $\partial_t$ on the mapping torus $M_\Phi$ does as well. But by Lemma \ref{lem:weinsteinStableHamiltonian}, the vector field $\partial_t$ does have a periodic orbit, so $X$ has a periodic orbit as well. Since we assumed that $X$ was uniquely ergodic, $X$ does not have a periodic orbit by assumption (see Example \ref{exe:invariantMeasureOrbit}). Therefore, we arrive at a contradiction and $X$ cannot be uniquely ergodic in this case as well.
\end{proof}

To conclude, we observe that we can only deduce non-unique ergodicity for manifolds that do not fiber over $S^1$ in the three-dimensional case, since Hutchings-Taubes' proof of the Weinstein conjecture for stable Hamiltonian structures is crucial to our argument and only works in dimension three. 

Therefore, Proposition \ref{prop:mainNonExactDim3} at the moment falls short of Corollary \ref{cor:uniqueErgodicity} in dimensions higher than three. However, it is clear that the exact same method extends to the following proposition. 

\begin{prop} \label{prop:alternateNonExact}
Let $(M^{2n+1}, \eta = (\lambda, \omega))$ be a framed Hamiltonian manifold such that:
\begin{enumerate}
    \item $\omega^n$ is not exact,
    \item The Weinstein conjecture for stable Hamiltonian structures holds true on $M$. That is, for any stable Hamiltonian structure $\eta' = (\lambda', \omega')$, the associated Hamiltonian vector field $X'$ has a periodic orbit.
\end{enumerate}

Then the Hamiltonian vector field $X$ is not uniquely ergodic. 
\end{prop}

\subsection{Proof of Proposition \ref{prop:actuallyAMappingTorus}} \label{subsec:mappingTorusProp}

We now give the previously deferred proof of Proposition \ref{prop:actuallyAMappingTorus}. Let $(M, \eta = (\lambda, \omega))$ be a closed framed Hamiltonian manifold such that $\omega^n$ is not exact and the Hamiltonian vector field $X$ is uniquely ergodic.

\begin{defn}
A \textbf{global cross section} of $X$ is an embedded surface $\Sigma$ in $X$ such that:
\begin{itemize}
    \item $\Sigma$ is transverse to $X$,
    \item $\Sigma$ intersects every orbit of the flow of $X$.
\end{itemize}

A \textbf{global hypersurface of section} of $X$ is a global cross section that intersects each orbit of the flow of $X$ infinitely many times in both forward and backward time. 
\end{defn}

We will first show that $X$ must have a global hypersurface of section.

\begin{prop} \label{prop:globalSurfaceOfSection}
Let $(M, \eta = (\lambda, \omega))$ be a closed framed Hamiltonian manifold such that $\omega^n$ is not exact and the Hamiltonian vector field $X$ is uniquely ergodic.
\end{prop}

The proof of this proposition relies on the following result due to Schwartzman \cite{Schwartzman57}, although the specific statement we give is taken from Theorem $2.27$ in \cite{Sullivan76}.

To state this result, we need to define the \emph{homology class} associated to an invariant measure $\sigma$ of $X$. 

This is done using the notion of a \emph{current}. For the definition of a \emph{$m$-dimensional current} and a \emph{closed $m$-dimensional current}, see Definition \ref{defn:currents}. 

There is a natural map from the space of closed $m$-dimensional currents into $H_m(M; \mathbb{R})$ as follows. First, note that since closed $m$-dimensional currents vanish on exact $m$-forms, their restriction to closed $m$-forms descends to a linear map $H^m(M; \mathbb{R}) \to \mathbb{R}$ on the degree $m$ de Rham cohomology group.

Then we obtain a homology class by the isomorphism
$$H_m(M; \mathbb{R}) \simeq \text{Hom}(H^m(M; \mathbb{R}), \mathbb{R}).$$

Finally, let $\sigma$ be an invariant measure of $X$. We associate to it a closed, one-dimensional current $T$ by setting, for any one-form $\gamma$,
$$T(\gamma) = \sigma(\gamma(X)).$$

To verify that $T$ is closed, note that for any exact one-form $df$, 
$$T(df) = \sigma(df(X)) = 0. $$

The second equality follows from the fact that $\sigma$ is invariant. Now we can state the required result of \cite{Schwartzman57}. We use the statement from \cite[Theorem II$.27$]{Sullivan76}. 

\begin{prop} \label{prop:asymptoticCycles}
\cite{Schwartzman57, Sullivan76} Suppose every non-trivial invariant measure of $X$ has non-zero homology class in $H_1(M; \mathbb{R})$. Then $X$ has a global cross section.
\end{prop}

Now we can prove Proposition \ref{prop:globalSurfaceOfSection}. We first show that, if $X$ is uniquely ergodic, then we can apply Proposition \ref{prop:asymptoticCycles} to obtain a global cross section. Then, we argue that this global cross section is a global hypersurface of section, again using the assumption that $X$ is uniquely ergodic. 

\begin{proof}[Proof of Proposition \ref{prop:globalSurfaceOfSection}]

Suppose that $X$ is uniquely ergodic. Then the only non-trivial invariant measure is the volume measure $\lambda \wedge \omega^n$. 

It is easy to see that the associated current to this measure is the current $T$ given by 
$$T(\gamma) = \int_M \gamma \wedge \omega^n.$$

The homology class of $T$ is simply the Poincar\'e dual to $\omega^n$, which is non-zero by our assumption that $\omega^n$ is not exact.

It follows that every non-trivial invariant measure of $X$ has non-zero homology class, so $X$ has a global cross section $\Sigma$ by Proposition \ref{prop:asymptoticCycles}. 

We claim that any orbit of $X$ must intersect $\Sigma$ infinitely many times in both forward and backward time.

First, observe that since $\Sigma$ is transverse to the flow of $X$, there is an open collar neighborhood $U$ of $\Sigma$ satisfying the following property. First, we can decompose $U \setminus \Sigma$ into two connected components $U_+$ and $U_-$. Second, for any point $p \in U_+$, the forward orbit of $X$ starting at $p$ intersects $\Sigma$. Third, for any point $p \in U_-$, the backward orbit of $X$ starting at $p$ intersects $\Sigma$. 

Now since $X$ is uniquely ergodic, for any point $p \in M$, both the forward and backward orbits of $X$ starting at $p$ must be dense in $M$. If not, we can average over the forward or backward orbit of $X$ starting at $p$ to obtain an invariant measure of $X$ that is not fully supported on $M$. 

It follows that, for any point $p$, the forward orbit of $X$ starting at $p$ intersects the open set $U_+$, which by definition implies it intersects $\Sigma$. Similarly, the backward orbit of $X$ starting at $p$ intersects the open set $U_-$, which by definition implies it intersects $\Sigma$.

We have therefore shown that, if $X$ is uniquely ergodic, $\Sigma$ is a global hypersurface of section.
\end{proof}

Now Proposition \ref{prop:actuallyAMappingTorus} is an easy corollary of Proposition \ref{prop:globalSurfaceOfSection}. 

\begin{proof}[Proof of Proposition \ref{prop:actuallyAMappingTorus}]
Let $\Sigma$ be the global hypersurface of section of $X$ in $M$. Then there is a smooth function 
$$\mathcal{R}_1: \Sigma \to \mathbb{R}$$
called the \emph{first return time}, which we define as follows.

Denote the flow of $X$ by the one-parameter family of diffeomorphisms $\phi^t_X$. At any point $p \in \Sigma$, $\mathcal{R}_1(p)$ is defined to be the smallest $t > 0$ such that 
$$\phi_X^{t}(p) \in \Sigma.$$

The first return time defines a natural diffeomorphism 
$$\Phi: \Sigma \to \Sigma$$
by
$$\Phi(p) = \phi^{\mathcal{R}_1(p)}_X(p).$$

The diffeomorphism $\Phi$ is \emph{volume-preserving} in the following sense. Note that since $X$ is transverse to $\Sigma$, the form $\omega^n$  on $M$ restricts to a volume form on $\Sigma$. Moreover, since the flow of $X$ preserves $\omega^n$, it is immediate that $\Phi$ preserves $\omega^n$. 

Let $M_\Phi$ be the mapping torus of $\Phi$. We define a diffeomorphism
$$h: M_\Phi \to M$$
by
$$h(t, p) = \phi^{t\mathcal{R}_1(p)}_X(p).$$

It is clear by construction that the pushforward $h_*\partial_t$ at the point $h(t,p)$ is equal to $\mathcal{R}_1(p) \cdot X$, so the proof of Proposition \ref{prop:actuallyAMappingTorus} is complete. 
\end{proof}

\bibliographystyle{halpha}
\bibliography{main}

\end{document}